\documentclass[dvips]{article}
\usepackage{amsmath,amssymb,amsfonts,amsthm,epsfig,latexsym,makeidx,graphics,titlesec}
\usepackage[usenames,dvipsnames]{color}
\newtheorem{theorem}{{Theorem}}[section]
\newtheorem{proposition}[theorem]{{Proposition}}

\newtheorem{lemma}[theorem]{{Lemma}}

\newtheorem{corollary}[theorem]{{Corollary}}

\newtheorem{question}[theorem]{{Question}}

\theoremstyle{definition}
\newtheorem{definition}[theorem]{{Definition}}

\theoremstyle{remark}
\newtheorem{remark}[theorem]{{Remark}}

\newtheorem{example}[theorem]{{Example}}

\title{Depth $0$ Nonsingular Morse Smale flows on $S^3$}
\author{Bin Yu}
\date{\today}

\begin{document}
\maketitle

\begin{abstract}
In this paper, we first develope the concept of Lyapunov graph to weighted Lyapunov graph (abbreviated as WLG) for nonsingular Morse-Smale flows (abbreviated as NMS flows) on $S^3$. WLG is quite sensitive to  NMS flows on $S^3$. For instance, WLG detect the indexed links of NMS flows. Then we use WLG and some other tools to describe nonsingular Morse-Smale flows without heteroclinic trajectories connecting saddle orbits
(abbreviated as depth $0$ NMS flows). It mainly contains the following several directions:
\begin{enumerate}
  \item we use WLG to list  depth $0$ NMS flows on $S^3$;
  \item with the help of WLG, comparing with  Wada's algorithm, we  provide a direct description about the (indexed) link of depth $0$ NMS flows;
  \item to overcome the weakness  that WLG can't decide topologically equivalent class, we give a simplified Umanskii Theorem to decide when two depth $0$ NMS flows on $S^3$ are topological equivalence;
   \item  under these theories, we classify (up to topological equivalence) all depth 0 NMS flows on $S^3$ with periodic orbits number no more than 4.
\end{enumerate}
\end{abstract}
\tableofcontents

\section{Introduction}\label{introduction}

\subsection{Historic remarks and the aim of the paper}
    Morse-Smale system (both of diffeomorphisms and flows) is a kind of structure stable system whose non-wandering set is composed of finitely many hyperbolic periodic orbits.
     Some historic remarks can be found in  \cite{BGL} and \cite{GP}.
     From the viewpoint of dynamics, it is more or less a kind of simple system. For instance,  there doesn't exist
     a homoclinic orbit in  a Morse-Smale flow. Roughly say, such a system doesn't provide chaos.
     But from the viewpoint more close to topology, i.e., classification (up to topological equivalence), it is quite complicated.

M. Peixoto \cite{Pe} begun  a systematical classification of  Morse-Smale system. In that paper, he focus on the classification of Morse-Smale flows on surfaces. 

An excellent pioneer about the  study of Morse-Smale flows, in particular,
  nonsingular Morse-Smale flows (abbreviated as NMS flows) on high dimension manifolds (dimension no less than $3$) is D. Asimov (\cite{As1}, \cite{As2}).
 His viewpoint is topology. Similar to the relationship between handle decompositions and gradient-like flows (Morse-Smale flows without closed orbits), he showed that there hints a combinatorial decomposition in an NMS flow, which is called a round handle decomposition (abbreviated as RHD). He used RHD to obtain two significant results about NMS flows on $n$-manifolds ($n\geq 4$).
 One (\cite{As1}) says that a  closed dimension $n$ manifold $M$ admitting an NMS flow if and only if the Euler number of $M$ is zero.
 Later, J. Morgan (\cite{Mo}) builds three theorems to nearly describe the 3-manifolds admitting NMS flows. For instance, one of his theorems says that an irreducible 3-manifold $M$ admits an NMS flow if and only if $M$ is a graph manifold.

 On the other hand, Y. Umanskii (\cite{Um}) and A. Prishlyak (\cite{Pr})  used some combinatorial objects as invariant to completely characterize  Morse-Smale flows. Umanskii used the combinatorial object: scheme.  A scheme is the family of all cells, i.e. components of $M$ {closures of separatrices of saddle orbits}, together with combinatorial information on sinks, sources and separatrices constituting the boundary of each cell. Prishlyak (\cite{Pr}) used $\tau$ invariant and framed graph to describe NMS flows.

However, the following natural question  nearly isn't involved in the works above.

 \begin{question} \label{basicquest}
 For a given 3-manifold $M$, how to describe the NMS flows on $M$?
 \end{question}

 M. Wada \cite{Wa} builds an algorithm to decide which (indexed) links   can be as periodic orbits of NMS flows on $S^3$. This work can be regarded
 as a proceed on Question \ref{basicquest}.  However, one is not easy to construct an NMS flow with a given indexed link.

 First, as we have mentioned, heteroclinic trajectories connecting saddle orbits will lead the discussion of NMS flows quite wild. Therefore, in this paper, we just deal with NMS flows without heteroclinic trajectories connecting saddle orbits,
which is called depth $0$ NMS flows.  Actually, the name makes sense in the viewpoint of Smale diagram. For some information about Smale diagram, we suggest \cite{BGL}.   Secondly, for simplify, we only discuss the case $M=S^3$. Although our discussion is restricted on $S^3$, as you will see, the tricks in the paper can be developed to study NMS flows, in particular, depth $0$ NMS flows on general 3-manifolds.

 The main purpose of this paper is to discuss Question \ref{basicquest} for depth $0$ NMS flows on $S^3$. More precisely, we hope to provide a convenient list for depth $0$ NMS flows on $S^3$. By the way, comparing with the work of Wada (\cite{Wa}), we expect to provide a more direct way to understand the (indexed) link of the NMS flows in the list.
 Furthermore, one naturally hope to know how to distinguish the flows in the list up to topological equivalence. We will discuss this
problem based on Umanskii Theorem in \cite{Um}. Finally, we will illustrate the theories  by some concrete examples, i.e., complete classification of Depth $0$ NMS flows with low number of periodic orbits .

\subsection{Main results and main tools}
 In this paper, to describe an NMS flow, we  use Lyapunov function  to  decompose and combine the flow as follows.
  Firstly, we cut the flow along some regular level sets of a Lyapunov function  to some canonical pieces, i.e.,
 filtrating neighborhoods. Notice that a filtrating neighborhood are called a fat round handle by Morgan \cite{Mo}.
  Then we analyze the possible canonical pieces, see Proposition \ref{filnghd}. Notice that this  analysis is already done by Morgan and it is fundamental for understanding NMS flows on 3-manifolds for all of the other researchers on this area.

  Furthermore, to record how to combine the filtrating neighborhoods to form the original flow, Franks \cite{Fr} introduced the concept Lyapunov graph. Roughly say, a Lyapunov graph of a flow is constructed as follows. For an NMS flow described in the last paragraph, the Lyapunov graph is a oriented compact graph whose
vertices and  edges are corresponding to filtrating neighborhoods and the regular level sets accordingly. Franks, K. de Rezende and the author (\cite{Fr}, [Re] and [Yu2]) used Lyapunov graphs to describe a more general kind of flows, i.e., Smale flows. Actually, Theorem 2 of Franks in \cite{Fr} exactly determines necessary and sufficient conditions on Lyapunov graph to be associated with an NMS flow on $S^3$. However, such a list in some sense is quite rough.
For instance, a Lyapunov graph with three vertices  admits infinitely many NMS flows such that the link types of their periodic orbits are pairwise different, see for instance, \cite{Sa} or Section \ref{3per} of the paper.

In this paper,  for recording more information of an NMS flow, we  refine Lyapunov graph to weighted lyapunov graph (abbreviated as WLG).
Now let's briefly introduce WLG. The definition about a WLG of an NMS flow $\varphi_t$ on $S^3$ is based on a Lyapunov graph $G$ of $\varphi_t$.
But there are two new weights:
\begin{enumerate}
  \item we label every vertex $V$ of $G$ by the corresponding filtrating neighborhood;
  \item we  endow each edge $e$ of $G$ by a matrix $A \in PSL(2,\mathbb{Z})$ to represent the gluing homeomorphism (up to isotopy) between the two filtrating neighborhoods associated to $e$.
\end{enumerate}
To define the second kind of weight,  we define some coordinates for the boundary of each filtrating neighborhood, see Section \ref{coordinate}. Notice that, in our case, a boundary component of a filtrating neighborhood is homeomorphic to a torus. This fact ensures that we can use a  matrix $A \in PSL(2,\mathbb{Z})$ to represent the gluing homeomorphism (up to isotopy) of an edge in $G$.

 After the definition of WLG of NMS flows. We turn to the topic about describing depth $0$ NMS flows by using WLG. First of all, we have the following useful observation. If an NMS flow contains one of  the first three types of  filtrating neighborhoods in Proposition \ref{filnghd}, one can split
  the flow to some  simpler NMS flows. On the other hand, this process in some sense is invertible. Such a surgery is called a flow split.
After a flow split, one can easily build WLG for the new NMS flows associated to  a WLG of the original NMS flow. This surgery among WLG is called Graph split. To understand depth $0$ NMS flows on $S^3$, split surgeries promise us that we only need to focus on the depth $0$ NMS flows with WLG which can't be done any split surgeries (called by simple WLG). All topics in this paragraph  can be found in  Section \ref{split}.

  Then we use WLG to list the possible depth $0$ NMS flows on $S^3$. For more precisely represent depth $0$ NMS flows on $S^3$,   we introduce a subset of all WLG, i.e., neat WLG. Then we use neat WLG represent depth $0$ NMS flows on $S^3$: Theorem \ref{simpleWLG} and Theorem \ref{generalWLG}. All topics about this part  can be found in  Section \ref{WLGD}.

  In Section \ref{link}, we discuss the indexed links of depth $0$ NMS flows on $S^3$. Actually, WLG decides the indexed link of NMS flows (Proposition \ref{decidelink}). Then we focus on the following question (Question \ref{Qindlink}):
   for a given neat WLG $G$ which satisfies the two conditions of Theorem \ref{generalWLG}, what is the indexed link of $G$?  The results in Proposition \ref {Wadaop}, Theorem \ref{noV5link} and Theorem \ref{V5link} answer this question.  By the way, B. Campos and P. Vindel \cite{CV} discussed the indexed links of a special class of Depth $0$ NMS flows, i.e., the depth $0$ NMS flows exactly with the first three types of  filtrating neighborhoods in Proposition \ref{filnghd} on $S^3$  by using Wada's algorithm.
 As a consequence, we can use ``simple depth $0$ NMS flow" substitute the concept ``depth $0$ NMS flow with simple WLG", see Question \ref{simpleflow} and Corollary \ref{anwq}.

   Maybe two different depth $0$ NMS flows admit the same WLG. In Section \ref{Umanskii}, we build a criterion (Theorem \ref{topequ}) to decide when two depth $0$ NMS flows on $S^3$ are topologically equivalent. Although the combinatorial tools what we use are different to what Umanskii (\cite{Um}) used, but essentially, it is a simplified Umanskii Theorem. More explanations about this also can be found in Section \ref{Umanskii}.
   As a consequence (Corollary \ref{finite}) of Theorem \ref{topequ},  a WLG (in particular, a neat WLG) associated to a depth $0$ NMS flow on $S^3$
   always admits finitely many depth $0$ NMS flows on $S^3$.

Under these theories, in the end of the paper (Section \ref{example}), we classify (up to topological equivalence) all depth 0 NMS
flows on $S^3$ with periodic orbits number no more than 4. In the cases periodic orbits numbers $2$ and $3$, we collect the classifications to
several propositions (Proposition \ref{2periorb}, Proposition \ref{3regper} and Proposition \ref{3twper}).  In the cases periodic orbits number $4$,
since the parameters are complicated, except for proposition \ref{4perV2}, we give a soft representation for the classification (see Section \ref{4per}). Notice that
 in the cases periodic orbits numbers $2$ and $3$, every NMS flow is a depth $0$ NMS flow. As a byproduct of this part, in \cite{Yu3}, the author use Proposition \ref{2periorb}, Proposition \ref{3regper} and  Proposition \ref{3twper} to represent the homotopy classes of nonsingular vector fields on $S^3$.

\section{Preliminaries}
\subsection{fundamental definitions and facts}
\begin{definition}
A smooth flow $\phi_t$ is called  a nonsingular Morse Smale flow (abbreviated as NMS flow) if it satisfies the following conditions:
\begin{enumerate}
  \item the non-wandering set of $\phi_t$ is composed of finitely many periodic orbits without singularity;
  \item each periodic orbit of $\phi_t$ is hyperbolic, i.e., the Poincare map for each periodic orbits is hyperbolic;
  \item the stable and unstable manifolds of periodic orbits intersect transversally.
\end{enumerate}
\end{definition}

\begin{definition}
A smooth flow is called  a depth $0$ nonsingular Morse Smale flow (abbreviated as a depth $0$ NMS flow) if the flow is an NMS flow without heteroclinic trajectories connecting saddle orbits.
\end{definition}

\begin{remark}
A heteroclinic trajectory $\gamma$ connecting saddle orbits is exactly a trajectory of a point in the intersection of the stable manifold and the unstable manifold of two saddle periodic orbit $\gamma_1$ and $\gamma_2$ correspondingly.
\end{remark}

\begin{definition}
A filtrating neighborhood $N(\gamma)$ of an isolated periodic  orbit $\gamma$ is a compact 3-manifold with vector field $X$ such that:
\begin{enumerate}
  \item $\gamma$ is the maximal invariant set of $(N(\gamma), X)$; and
  \item $X$ is transverse to $\partial N(\gamma)$ and each flowline is connected.
\end{enumerate}
\end{definition}

The definition and the existence of Lyapunov function for a smooth flow, i.e.,
the fundamental theorem of dynamical system, can be found in many standard books, for instance, \cite{Rob}.
Based on the existence of Lyapunov function, Franks \cite{Fr} defined Lyapunov graphs to combinatorially study some flows.
For us, we restrict such a tool in NMS flows.

\begin{definition}
An \emph{abstract Lyapunov graph} is a
finite, connected, oriented graph $L$ which possesses no oriented
cycles, and each vertex of which is labeled with an isolated hyperbolic periodic orbit.
\end{definition}

\begin{definition}
A Lyapunov graph $L$ for an NMS flow $\phi_t$ and a Lyapunov function $f:M\rightarrow R$ is
obtained by taking the quotient complex of $M$ by identifying to a
point each component of a level set of $f$.
\end{definition}

\begin{remark}
If $L$ is a Lyapunov graph of an NMS flow $\phi_t$ on a 3-manifold $M$, we choose a regular point $v_e$ in each edge $e$ of $L$. Suppose $T_e$ is the natural preimage of $v_e$ in $M$. All these preimages cut $M$ to several pieces. $\phi_t$ restricts to each piece  is a filtrating neighborhood with
a periodic orbit as the maximal invariant set (maybe we need to combine some $T^2 \times I$ pieces to the other pieces).  This fact provides us a natural way to understand NMS flows as follows. Firstly, we describe all possibilities of filtrating neighborhoods. Then we discuss how these filtrating neighborhoods can be glued together.
\end{remark}

\begin{definition}
Two smooth flows $\phi_1$ and $\phi_2$ on a manifold $M$ are called \emph{topologically equivalent} if there exists a homeomorphism $f:M\rightarrow M$
preserving trajectories. More precisely, $f$  sends  each trajectory  of $\phi_1$ to a trajectory of $\phi_2$ and preserves orientations.
\end{definition}

\begin{definition}
A link in a three manifold is called an \emph{indexed link} if each component of $L$ is oriented and labeled by an index which belongs to $\{0,1,2\}$.
\end{definition}

In particular,  the periodic orbits  of an NMS flow $\phi_t$ on a three manifold $M$ naturally provide an indexed link which is labeled as follows.
\begin{enumerate}
  \item The link is oriented by  the flow direction of $\phi_t$.
  \item The attractors, the saddle periodic orbits and the repellers of  $\phi_t$ are indexed by $0$, $1$ and $2$ accordingly.
\end{enumerate}

\subsection{filtrating neighborhoods}

Suppose
 $V=D^2 \times S^1=\{(z, t)\mid z\in \mathbb{C}, |z|\leq 1, t=e^{i\theta} \in S^1\}$ ($\theta \in [0,2\pi]$) with vector field
$X=(\dot{z}, \dot{\theta})=(\overline{z}, 1)$,
 and $\widetilde{V}=D^2 \times S^1=V$ with vector field $\widetilde{X}=(\dot{z}, \dot{\theta})=(e^{i\frac{\theta}{2}} \overline{z}, 1)$.

 Then,
 \begin{itemize}
   \item  $X|\partial V$ is divided to $4$ annuli components along $4$ dividing curves such that $X$ is transverse to the interior of these annuli;
   \item $\widetilde{X}|\partial \widetilde{V}$ is divided to $2$ annuli components along $2$ dividing curves such that $X$ is transverse to the interior of these annuli.
 \end{itemize}

 \begin{figure}[htp]
\begin{center}
  \includegraphics[totalheight=4cm]{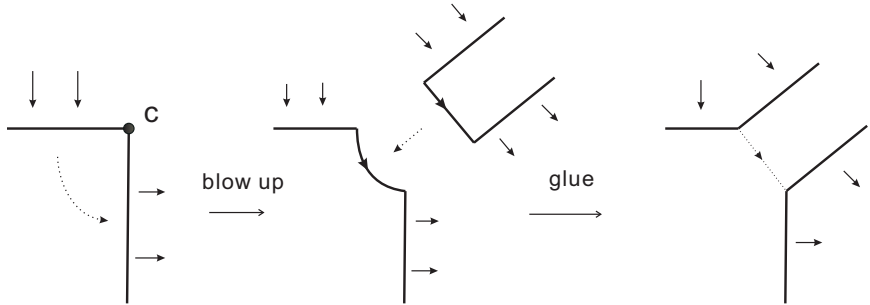}\\
  \caption{the sectional view of the attaching surgery}\label{attaching}
  \end{center}
\end{figure}

 The following lemma is standard, see, for instance, \cite{PD}.

\begin{lemma}
Let $X$ be a smooth vector field on an orientable 3-manifold $M$ and $\gamma$ be a saddle periodic orbit of $X$, then there exists a tubular neighborhood $V(\gamma)$ of $\gamma$ such that $(V(\gamma), X|)$ is topologically equivalent to either $V$ or $\widetilde{V}$. We call $\gamma$ a \emph{normal saddle periodic orbit} if $\gamma$ has a neighborhood which is topologically equivalent to $V$; otherwise, we call $\gamma$ a \emph{twisted saddle periodic orbit}.
\end{lemma}

Suppose $\Sigma$ is a compact orientable surface, we call $\Sigma \times [0,1]$ with vector field $\frac{\partial}{\partial t}$ a \emph{thickened surface}. We can attach a thickened surface $\Sigma \times [0,1]$ to $V$ or $\widetilde{V}$ along a dividing curve $c$ as follows. Firstly we can blow up $c$ to $c\times [0,1]$ with  vector field $\frac{\partial}{\partial t}$. Then we choose a boundary $c_0$ of $\Sigma$ and glue $c_0 \times [0,1]$ to $c \times [0,1]$ by preserving the flowlines. See Figure \ref{attaching}.
The following lemma is a special case of Theorem 4.4 in \cite{Yu2}.

\begin{lemma} \label{attiksuf}
Let $X$ be a smooth vector field in a closed orientable 3-manifold $M$ and $\gamma$ be an isolated saddle periodic orbit of $X$, then a filtrating neighborhood $N(\gamma)$ of $\gamma$ always can be obtained by attaching some thickened surfaces along all dividing curves of $V$ or $\widetilde{V}$.
\end{lemma}

The following proposition is fundamental  in this paper. The main part   was firstly stated by Morgan and Wada (\cite{Mo}, \cite{Wa}). But for stating more clearly and insisting on using Lyapunov function to understand flows, we reprove it by using the two lemmas above.  Moveover,  we state more about the positions of the stable and unstable manifolds.

\begin{figure}[htp]
\begin{center}
  \includegraphics[totalheight=4.2cm]{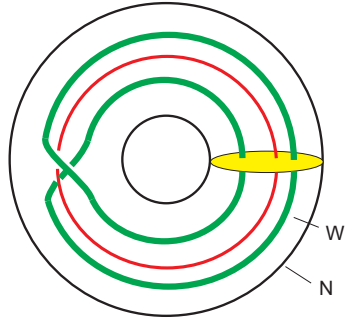}\\
  \caption{the 5th filtrating neighborhood}\label{V5}
  \end{center}
\end{figure}

\begin{proposition} \label{filnghd}
A filtrating neighborhood $N$ of a  saddle periodic orbit $\gamma$ in an NMS flow on $S^3$ is one of the following five cases.
\begin{enumerate}
  \item $N$ is homeomorphic to $(T_1 \times [0,1])\sharp (T_2 \times [0,1])$ where $T_1$ and $T_2$ are two tori. $\partial^{in} N = (T_1 \times \{0\})\cup (T_2 \times \{0\})$ and $\partial^{out} N = (T_1 \times \{1\})\cup (T_2 \times \{1\})$.
      $\gamma$ is a trivial knot in $N$.
   $W_1^s (\gamma)$ is inessential in $T_1 \times \{0\}$ and $W_2^s (\gamma)$ is inessential in $T_2 \times \{0\}$.
      It is similar to $\partial^{out} N$ and $W^u (\gamma)$.

  \item $N$ is homeomorphic to $T^2 \times [0,1] \sharp D^2 \times S^1$.
  $\partial^{in} N = (T^2 \times \{0\})$ and $\partial^{out} N = (T^2 \times \{1\})\cup \partial D^2 \times S^1$. $\gamma$ is a trivial knot in $N$.
   $W_1^s (\gamma)$ is inessential in $T^2 \times \{0\}$ and $W_2^s (\gamma)$ is a meridian  in $D^2 \times S^1$.
  $W^u (\gamma)$ is composed of two inessential simple closed curves bounding an annulus with the same orientation.
 Or $N$ is topologically equivalent to the flow above by changing the orientation.

  \item $N$ is homeomorphic to $D_1 \times S^1\sharp D_2 \times S^1$ where $D_1$ and $D_2$ are two disks. $\partial^{in} N = \partial D_1 \times S^1$ and $\partial^{out} N = \partial D_2 \times S^1$.
      $\gamma$ is a trivial knot in $N$.
   $W_1^s (\gamma)$  is  a meridian  and $W_2^s (\gamma)$ is inessential in $\partial^{in} N$.
    It is similar to $\partial^{out} N$ and $W^u (\gamma)$.

  \item $N$ is homeomorphic to $F \times S^1$ where $F$ is a disk with two holes. Suppose $\partial N = T_1 \cup T_2 \cup T_3$.  $\partial^{in} N = T_1 \cup T_2$ and $\partial^{out} N= T_3$. $W_1^s (\gamma) \subset T_1$, $W_2^s (\gamma) \subset T_2$ and $W^u (\gamma) \subset T_3$. Furthermore,
      $\gamma$, $W^s (\gamma)$ and $W^u (\gamma)$ are parallel to $\gamma$ with the same orientation.
 Or $N$ is topologically equivalent to the flow above by changing the orientation.

  \item $N$ is homeomorphic to $D^2 \times S^1 - W$
    where $W$ is a tubular neighborhood of a $(2, 1)$-cable knot of
$\{0\} \times S^1$ in the interior of $D^2\times S^1$,  $\partial^{in} N = \partial D^2 \times S^1$ and $\partial^{out} N = \partial W$. $\gamma=\{0\}\times S^{1}$ and $W^s (\gamma)$ is an essential simple closed curve intersecting with a meridian 2 times. It is similar to $W^u (\gamma)$. See Figure \ref{V5}.
\end{enumerate}
\end{proposition}

\begin{proof}
By Lemma \ref{attiksuf}, we only need to discuss how to attach thickened surfaces to form a filtrating neighborhoods of an NMS flow on $S^3$. There are two restrictions:
\begin{enumerate}
  \item each connected component of the boundary of such a filtrating neighborhood is homeomorphic to a torus;
  \item obviously, such a filtrating neighborhood can be embedded into $S^3$.
\end{enumerate}
Then we exactly have the following $5$ possibilities which correspond to the five cases listed in the proposition. Case $1$, case $2$, case $3$  and case $4$  correspond to $V$; case $5$ corresponds to $\widetilde{V}$.
\begin{enumerate}
  \item $(c_1 + (T^2 -D^2))\cup (c_2 +D^2)\cup (c_3 +(T^2 -D^2))\cup (c_4 +D^2)$ corresponds to the first case in the list of the proposition.
  Here, $c_1 + (T^2 -D^2)$ means to attach the thickened $(T^2 -D^2)$ to $V$ along $c_1$.
  \item $(c_1, c_2 +S^1 \times [0,1]) \cup (c_3 +(T^2 -D^2))\cup (c_4 +D^2)$ corresponds to the second case in the list.
  \item $(c_1, c_2, c_3 +(S^2 -3D^2))\cup (c_4 +D^2)$ corresponds to the third case in the list.
  \item $(c_1,c_2 +S^1 \times [0,1])\cup (c_3, c_4 + S^1 \times [0,1])$ corresponds to the 4th case in the list.
  \item $(c_1, c_2 + S^1 \times [0,1])$ corresponds to the 5th case in the list.
\end{enumerate}
\end{proof}

\begin{remark}\label{loctopequ}
 Notice that the proposition doesn't completely classify filtrating neighborhoods  of a  saddle periodic orbit of NMS flows on $S^3$.
  Actually, if we pay more attentions to the gluing in the proof, we can easily refine the proposition to the complete classification as follows.
  \begin{enumerate}
    \item In the first case, there are two topologically equivalent classes: both of  $W_1^s (\gamma)$ and $W_2^s (\gamma)$ are left-hand or right-hand orientations (associated to the direction of the vector fields) in $\partial^{in} N$.
    \item In the second case, each subcase (depending on the orientation of the vector field) implies two topologically equivalent classes. For instance, the subcase that $\partial^{in} N = (T^2 \times \{0\})$ and $\partial^{out} N = (T^2 \times \{1\})\cup \partial D^2 \times S^1$ implies two topologically equivalent classes depending on  wether $W_1^s(\gamma)$ is left-hand  orientation or not in $\partial^{in} N$.
    \item In the third case, there are two topologically equivalent classes depending on wether $W_2^s (\gamma)$ is left-hand   orientation or not in $\partial^{in} N$.
    \item In the other cases, each case (or subcase) is one and one corresponding to a topologically equivalent class.
  \end{enumerate}
 \end{remark}

 \begin{remark}
 If we use an irreducible 3-manifold $M$ to instead of $S^3$ and discuss the possible filtrating neighborhoods of a saddle periodic orbit of an NMS flow on $M$ , the first restriction in the proof still exists, see \cite{Mo}. Except for the five cases in the proposition, there is a new one: $(c_1, c_3 + S^1 \times [0,1])\cup (c_2, c_4 + S^1 \times [0,1])$, i.e., the corresponding filtrating neighborhood  is homeomorphic to a  twisted I-bundle over a punctured Mobius band.
\end{remark}

\subsection{definition of weighted Lyapunov graph} \label{coordinate}
We endow the boundary of the filtrating neighborhoods with some coordinates as follows.

\begin{itemize}

  \item If $N$ has a component (in the prime decomposition of $N$) $V=D^2 \times S^1$, then there exists a meridian $m \subset \partial V$ which belongs to $W^s (\gamma)$ or $W^u (\gamma)$, then the orientation of $\gamma$ naturally endows $m$ an orientation. We  can choose a simple closed curve $l$ with an orientation  (it is called a longitude) such that:
         \begin{itemize}
           \item the  intersection number of $l$ and $m$ is $1$;
           \item $(l,m, v)$ is right-hand orientation  where $v$ is a vector in $T_{p} V$ ($p \in m$) corresponding to the flow on $N$.
         \end{itemize}
       In this case, if we choose another longitude $l'$, then
       $\left(
         \begin{array}{c}
           l' \\
           m \\
         \end{array}
       \right)$ =  $\left(
                     \begin{array}{cc}
                       1 & k \\
                       0 & 1 \\
                     \end{array}
                   \right)$ $\left(
                               \begin{array}{c}
                                 l \\
                                 m \\
                               \end{array}
                             \right)
                   $

  \item If $N$ has a component $W=T^2 \times [0,1]$, then we can choose simple closed curves $l_i, m_i \subset T^2 \times \{i\}$ ($i=0,1$)   with orientations, called a longitude and a meridian correspondingly, which satisfies the following conditions to form coordinates correspondingly.
      \begin{itemize}
        \item The intersection number between $l_i$ and $m_i$ ($i=0,1$) is $1$.
        \item $l_0$ and $l_1$ are parallel  and  have the same orientation in $W$; $m_0$ and $m_1$ are parallel  and  have the same orientation in $W$.
        \item $(l_i,m_i, v_i)$ ($i=0,1$) is right-hand orientation where $v_i$ is a vector in $T_{p_i} V$ ($p_i \in T^2 \times \{i\}$) corresponding to the flow on $N$.
      \end{itemize}
      In this case, if we choose another coordinates $l_i', m_i'$ ($i=0,1$), then there exists a matrix $P \in PSL(2,\mathbb{Z})$ such that
       $\left(
         \begin{array}{c}
           l_i' \\
           m_i' \\
         \end{array}
       \right)$ =  $P$ $\left(
                               \begin{array}{c}
                                 l_i \\
                                 m_i \\
                               \end{array}
                             \right)
                   $.

  \item If $N$ is homeomorphic to $F \times S^1$ as case 4 in Proposition \ref{filnghd}. We can choose two simple closed curves $l_i, m_i \subset T_i$ ($i=1,2,3$) with orientations, called a longitude and a meridian correspondingly, which satisfies the following conditions to form coordinates correspondingly.
      \begin{itemize}
        \item $m_1, m_2$ and $m_3$ bound a disk with two holes in $N$.
        \item $l_i$ ($i=1,2,3$) is parallel to  the corresponding invariant manifolds in $T_i$. Moreover, $l_i$ has the same orientation with the corresponding invariant manifolds.
        \item $(l_i,m_i, v_i)$ ($i=1,2,3$) is right-hand orientation where $v_i$ is a vector in $T_{p_i} V$ ($p_i \in T_i$) corresponding to the flow on $N$.
      \end{itemize}

  \item If $N$ is homeomorphic to $D^2 \times S^1 - W$ as case 5 in Proposition \ref{filnghd}. $\partial D^2 \times S^1$  and $\partial W$ are defined by $T_1$ and $T_2$ correspondingly.  We can choose simple closed curves $l_i, m_i \subset T_i$ ($i=1,2$) with orientations, called a longitude and a meridian correspondingly, which satisfies the following conditions to form coordinates correspondingly.
       \begin{itemize}

        \item $l_i$ ($i=1,2$) is parallel to  the corresponding invariant manifolds in $T_i$. Moreover, $l_i$ has the same orientation with the corresponding invariant manifolds.
         \item The geometrical intersection number of  $m_i$ and $l_i$ ($i=1,2$) is $1$. Moreover, $m_i$ has $0$ fiber direction element. (It is easy to prove that they are unique up to isotopy on the corresponding torus)
        \item $(l_i,m_i, v_i)$ ($i=1,2$) is right-hand orientation where $v_i$ is a vector in $T_{p_i} V$ ($p_i \in T_i$) corresponding to the flow on $N$.
      \end{itemize}

\end{itemize}

\begin{remark}\label{freecoor}
\begin{enumerate}
  \item By studying the topology of $F \times S^1$, we can see that $(m_1, m_2, m_3)$ in the 4th case aren't unique up to isotopy in the corresponding tori. There are   $\mathbb{Z} \times \mathbb{Z}$ freedoms  for $(m_1, m_2, m_3)$ because of no restriction of longitude directions and the restriction that $m_1 \cup m_2 \cup m_3$ bounds pants.
  \item In the 5th case, it is obvious that $m_1$ and $m_2$ aren't unique up to isotopy in their corresponding tori.
\end{enumerate}

\end{remark}

\begin{definition}
An abstract \emph{weighted Lyapunov graph} $G$ is a compact oriented graph which is labeled by more things as follows.
\begin{enumerate}
  \item Each vertex is indexed by one of $R$, $A$, $V_1$, $V_2$, $V_3$, $V_4$ and $V_5$. Moreover a  $V_i$ ($i\in \{1,\dots, 5\}$) vertex is called  a saddle vertex.
  \item If a vertex $V\in G$ is labeled by $V_1$ or $V_2$, then a germ of $V$, $g(V)$ is defined to be a small neighborhood of $V$.
  $g(V)$ is labeled as the label rules below.
  \item For every oriented edge $e \in G$, we correspond $e$ to a matrix $B_e \in PSL(2,\mathbb{Z})$, which is called a \emph{gluing matrix}.
\end{enumerate}
The signs for  vertices, the matrices for oriented edges  and the labeled germs of $G$ are called the weight of $G$.
A \emph{simple weighted Lyapunov graph} $G$ is a Lyapunov graph without vertex labeled by $V_1$, $V_2$ and $V_3$.
\end{definition}

The label rules can be defined as follows.
\begin{enumerate}
       \item If $V$ is labeled by $V_1$, then $g(V)$ is composed of one vertex and $4$ rays. Under the orientations of $G$, two rays begin at $V$ and  the other two rays terminate at $V$. Then we label  one ray starting at $V$ and one ray terminating at $V$ by red color. Similarly we label the other two rays by green color.
       \item If $V$ is labeled by $V_1$, then $g(V)$ is composed of one vertex and $3$ rays. Either two rays begin at $V$ and  the other ray terminates at $V$, or vice versa. In arbitrary case,  we label  one ray starting at $V$ and one ray terminating at $V$ by red color and label the other ray by green color.
\end{enumerate}

\begin{definition} \label{WLGequ}
We call $G_1$ and $G_2$   \emph{WLG equivalent} if there exists a homeomorphism $g: G_1 \rightarrow G_2$ preserving all the corresponding weights of the WLG. The map $g$ is called a WLG map.
\end{definition}

\begin{definition}
We say that a weighted Lyapunov graph (or simple Lyapunov graph) $G$ is corresponding to a nonsingular Morse Smale flow $\phi_t$ on $S^3$ if there exists a Lyapunov function $g$ of $\phi_t$ satisfying the following conditions.
\begin{enumerate}
  \item If we forget the weights of the vertices and the edges of $G$, $G$ is the Lyapunov graph of $\phi_t$ associated to $g$.
  \item $R$, $A$ and $V_i$ ($i\in \{1,2,3,4,5\}$) are corresponding to the filtrating neighborhoods of a repeller, an attractor  and case $i$ in Proposition \ref{filnghd} respectively.
  \item The red colored rays of $g(V_1)$ ($g(V_2)$) are associated to $T_1 \times \{0\}$ and $T_1 \times \{1\}$ (respectively, $T^2 \times \{0\}$ and $T^2 \times \{1\}$).
  \item For each boundary component of every filtrating neighborhood associated to the vertices of $G$ satisfying the following condition, we give a coordinate. Let $e$ be an oriented edge of $G$. Define $N_0$ and $N_1$ to be the  filtrating neighborhoods associated to beginning and terminal vertices of $e$ correspondingly. Define $T_0$ and $T_1$ to be the boundaries of $N_0$ and $N_1$ associated to $e$ under the restriction provided by the above condition. Set $(l_i, m_i)$ ($i=0,1$) is the coordinate of $T_i$,
      then the  matrix $B_e$ represents the gluing map from $T_0$ to $T_1$ under the coordinates.
\end{enumerate}
\end{definition}

\begin{remark}
In the following of the paper, generally, when two WLG $G_1$ and $G_2$  are associated  to NMS flows,  we regard them as the same if their difference only is because of the choosing of coordinates (see Remark \ref{freecoor}). But sometimes we need the sharpest form, i.e., WLG equivalence. When this case appears, we will  point it out.
\end{remark}

\subsection{two topological facts}
The following two lemmas are standard in low dimensional topology. They are useful in the following of the paper. The notations which we will use are standard in low dimensional topology. One can find them in Hatcher's book \cite{Ha}.

\begin{lemma}
Let $T$ be an embedded torus in $S^3$, then $T$  bounds a solid torus $V$ in $S^3$.
\end{lemma}

\begin{lemma} \label{top}
\begin{enumerate}
  \item $D^2 \times S^1$ is a Serfert manifold and its Seifert fiber structure exactly can be represented as $D^2 \times S^1 = M(0,1; \alpha/ \beta)$.
  \item $S^3$ is a Seifert manifold and its Seifert fiber structure exactly can be represented as

  $S^3 = M(0,0; \alpha_1 / \beta_1, \alpha_2 / \beta_2 )$.
\end{enumerate}
\end{lemma}

The proof and more details about the first lemma can be found in, for instance, \cite{Ro};   the proof of the second lemma can be found in   \cite{Ha} Theorem 2.3.

\section{Split surgeries}\label{split}

\subsection{flow split}

Let $\phi_t$ be an NMS flow on $S^3$, $N$ be a filtrating neighborhood of a saddle periodic orbit $\gamma$ of $\phi_t$.
If $N$ is a type 1, 2 or 3 filtrating neighborhood in Proposition 2.13, then we can do a kind of surgery to build two new NMS flows on $S^3$. We  call such kind of surgery \emph{flow split}.  Flow split has three forms depending on the type of the filtrating neighborhood $N$.
\begin{enumerate}
  \item \emph{Flow split $I$}: $N$ is described as 1 of Proposition 2.13, suppose $N_1^0 \cup N_1^1 \cup N_2^0 \cup N_2^1 = \overline{S^3 -N}$ where $\partial N_i^j =T_i \times \{j\}$ for every $i \in \{1,2\}$ and $j \in \{0,1\}$. Then we glue $N_i^0$ to $N_i^1$  along $T_i \times \{0\}$ and $T_i \times \{1\}$ such that  up to isotopy on  the two boundary tori, the gluing map is the same with the map induced by the flowlines in $N$. Moveover the map ensures that the stable and unstable manifolds of the saddle periodic orbits transversely intersect.   After this surgery, we obtain two NMS flows $\phi_t^1$ and $\phi_t^2$ on $S^3$.
  \item \emph{Flow split $II$}: $N$ is described as 2 of Proposition 2.13, $N_1$, $N_2$ and $N_0$ are associated to the three connected components of $\overline{S^3 -N}$ with boundary $T_1$, $T_2$ and $T_0$ correspondingly. Then we glue $N_1$ and $N_2$ together along $T_1$ and $T_2$  such that up to isotopy on  the two boundary tori, the gluing map is similar to the flow split $I$. Moveover, We glue a standard filtrating neighborhood of a periodic orbit attractor or repeller to $N_0$.    After this surgery, we obtain two NMS flows $\phi_t^1$ and $\phi_t^2$ on $S^3$. Notice that in this case $\phi_t^1$ and $\phi_t^2$ are not symmetric.
  \item \emph{Flow split $III$}: $N$ is described as 3 of Proposition 2.13,  $N_1$ and $N_2$ are associated to the two connected components of $\overline{S^3 -N}$ with boundary $T_1$ and $T_2$ correspondingly. Then we glue a standard filtrating neighborhood of a periodic orbit repeller and attractor to $N_1$ and $N_2$ to form NMS flows $\phi_t^1$ and $\phi_t^2$ on $S^3$ correspondingly.
\end{enumerate}

Notice that in flow split $I$ and flow split $II$, the new NMS flow isn't unique up to topological equivalence except for $\phi_t^2$ in the flow split $II$. The reason is that topologically equivalent class depends on the gluing  of stable and unstable manifolds of saddle periodic orbits on the gluing boundary. But each flow has an unique  indexed link.

\subsection{graph split}

Let $G$ be a WLG which is homeomorphic to a tree, then we can define \emph{graph split} along $V_1$, $V_2$ or $V_3$ to build two new WLG. See Figure \ref{graphsplit}.

\begin{figure}[htp]
\begin{center}
  \includegraphics[totalheight=3.9cm]{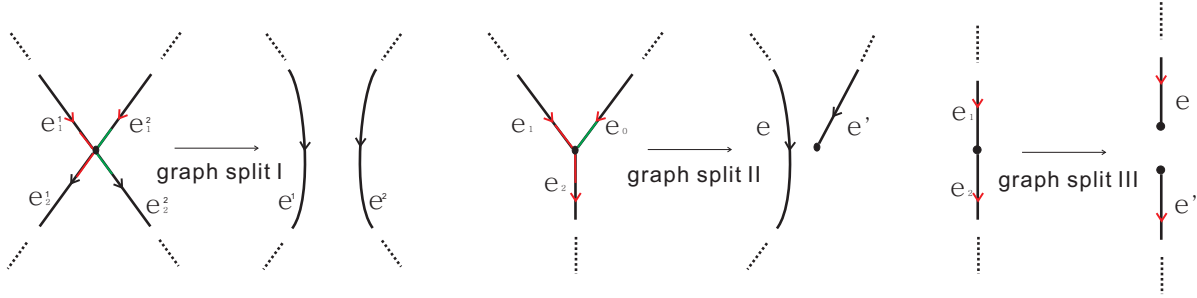}\\
  \caption{Graph split}\label{graphsplit}
  \end{center}
\end{figure}

\begin{enumerate}
  \item \emph{Graph split $I$}: Suppose $V$ is a vertex in $G$ labeled with $V_1$ and the germ of $V$, then we can build two new WLG $G_1$ and $G_2$ from $G$ as follows.
      \begin{enumerate}
        \item $G_1 \cup G_2$ is obtained by $G- U(V)$ connecting $V^0 (e_1^i)$ with $V^1 (e_2^i)$ ($i=1,2$) along a new oriented edge $e^i$.
      Here $G_i$ ($i=1,2$) contains $e^i$ and $U(V)$ is the union of $V$ with the edges disjoint to $V$.  For an edge $e\subset G$, $V^0 (e)$ and $V^1 (e)$ are the starting vertex and the terminal vertex of $e$ correspondingly.

        \item In $G_i$ ($i=1,2$),  except for the matrix $B_{e^i}$ for $e^i$,  the weights are the same to the corresponding weights of  $G$.
        \item $B_{e^i} = B_{e_1^i} \cdot B_{e_2^i}$.
      \end{enumerate}
  \item \emph{Graph split $II$}: Suppose $V$ is a vertex in $G$ labeled with $V_2$, then we can build two new WLG $G_1$ and $G_2$ from $G$ as follows.
      \begin{enumerate}
        \item $G_1 \cup G_2$ is obtained by $G- U(V)$ connecting $V^0 (e_1)$ with $V^1 (e_2)$  along a new oriented edge $e$ and gluing $V^0 (e_0)$ ($V^1 (e_0)$) to  an $R$ (respectively, $A$) vertex along a new edge $e'$. Here $e_1$ and $e_2$ are associated to the red rays of $g(V)$ respectively. The first connected component  is called $G_1$ and the second one is called $G_2$.
        \item In $G_i$ ($i=1,2$),  the weights are the same to the corresponding weights of $G$ except for the matrix $B_{e}$ for $e$ in $G_1$.
        \item $B_e = B_{e_1} \cdot B_{e_2}$.
      \end{enumerate}
  \item \emph{Graph split $III$}: Suppose $V$ is a vertex in $G$ labeled with $V_3$ and $V^1 (e_1)=V^0 (e_2)=V$. Then we can build two new WLG $G_1$ and $G_2$ from $G$ as  follows.
      \begin{enumerate}
        \item $G_1 \cup G_2$ is obtained by $G- U(V)$ gluing $e$ and an $A$ vertex to $V^0 (e_1)$ and gluing $e'$ and an $R$ vertex to $V^1 (e_2)$. The connected component containing $e$ is called $G_1$ and the other connected component is called $G_2$.
        \item In $G_i$ ($i=1,2$), if a vertex, the weights are the same to the corresponding weights of $G$. In particular, $B_e = B_{e_1}$ and $B_{e'} =B_{e_2}$.
      \end{enumerate}
\end{enumerate}

\subsection{simple WLG decomposition}

Flow splits and Graph splits have some natural relations as follows.
 Let $G$ be a WLG of an NMS flow $\phi_t$ on $S^3$.
 Suppose $V_i (i\in \{1,2,3\})$ is a filtrating neighborhood of $\phi_t$.
  Then we can do a flow split $j$ ($j=I$, if $i= 1$; $j=II$, if $i=2$; $j=III$, if $i=3$) at $V_i$ on $\phi_t$ to obtain two new NMS flows $\phi_t^1$ and $\phi_t^2$. Therefore, the new WLG $G_1$ and $G_2$ obtained by doing Graph split $j$  at the $V_i$  vertex are two WLG of $\phi_t^1$ and $\phi_t^2$ correspondingly. We call such a graph split  is associated to  the corresponding flow split.

\begin{proposition} \label{SWLGD}
Let $G$ be a WLG  which is homeomorphic to a tree. For a given order of all $V_1, V_2$ and $V_3$ vertices of $G$,  if we  do  graph splits one by one under this order on all  $V_1, V_2$ and $V_3$ vertices of $G$, we can obtain new   WLG  $G_1, ..., G_n$. Then $\{G_1, ..., G_n\}$ satisfies the following conditions.
\begin{enumerate}
  \item As a WLG set, $\{G_1, ..., G_n\}$ doesn't depend on the order;
  \item Each $G_i$ ($i\in\{ 1,..., n\}$) is a simple WLG.
\end{enumerate}
\end{proposition}

\begin{proof}
Suppose $V^1$ and $V^2$ are two vertices in $G$ such that $V^1$ and $V^2$ are $V_i$ ($i=1,2$ or $3$) vertices. The split surgeries on $V^1$ and $V^2$
are called  $S_1$ and $S_2$ correspondingly. Suppose $G \overset{S_1}{\rightarrow} G^1 \overset{S_2}{\rightarrow} G^{21}$ and $G \overset{S_2}{\rightarrow} G^2 \overset{S_1}{\rightarrow} G^{12}$. To prove $1$ of the proposition, we only need to check that $G^{21}=G^{12}$.

If $V^1$ and $V^2$ aren't adjacent in $G$, since $S_1$ and $S_2$ are local and independent, obviously  $G^{21}=G^{12}$.
Otherwise, $V^1$ and $V^2$ are adjacent in $G$, the associative law of the multiplication of matrices ensures that $G^{21}=G^{12}$.

Moreover, the fact that no new $V_1, V_2$ and $V_3$ vertices appear in the proceed of the surgery ensures that each $G_i$ is a simple WLG.
\end{proof}

\begin{definition}
We call each $G_i$ ($i \in\{1,...,N\}$)  a \emph{simple WLG factor} of $G$ and $\{G_1,\dots, G_n\}$ the \emph{simple WLG decomposition} of $G$. Proposition \ref{SWLGD} tell us that simple WLG decomposition of a WLG $G$ is well defined. An $A$ or $R$ vertex $V$ in a simple WLG factor $G_i$ of $G$ is called
\emph{a special vertex} if $V$ is decomposed from graph split III or $G_2$ of graph split II.
  Moreover, the graph which is obtained by cutting all $A$ and $R$ vertices of a simple WLG is called \emph{a saddle simple piece}.
\end{definition}

For an NMS flow $\phi_t$ on $S^3$ with a WLG $G$, suppose $\{G_1, \dots, G_n\}$ is the simple WLG decomposition of $G$. Then we can obtain NMS flows
$\phi_t^1, \dots, \phi_t^n$ by  the corresponding flow splits  to these graph splits. Naturally, the NMS flow $\phi_t^i$ ($i\in\{1,\dots,n\}$) is associated to the WLG $G_i$. Such a flow decomposition is called \emph{the flow splitting of $\phi_t$ associated to $G$}. Moreover, a sub-manifold
with restricted flow  corresponding to a saddle simple piece of the WLG is called \emph{a saddle simple piece of $\phi_t$ associated  to $G$}.

It is interesting to ask whether we can define simple WLG NMS flows. Let's state it more precisely.
\begin{question}\label{simpleflow}
 If an NMS flow $\phi_t$ on $S^3$ can be associated to a simple WLG, whether every WLG of $\phi_t$ is simple?
\end{question}

We will prove that the answer is positive for depth $0$ NMS flows (see Corollary \ref{anwq}).

\section{WLG of Depth $0$ NMS flows on $S^3$}\label{WLGD}
\subsection{Neat WLG}

Suppose $G$ is a WLG of some NMS flow on $S^3$ and $e_1$ and $e_2$  are two edges labeled with matrices $A$ and $B$ accordingly. Moreover, assume that $e_1$ and $e_2$ are adjacent to two $A$ or $R$ vertices $V^1$ and $V^2$ respectively.  It is obvious that there  exists an unique oriented path $l$ starting at $V^1$ and terminating at $V^2$. This path reorders $e_1$ and $e_2$ to $\overline{e_1}$ and $\overline{e_2}$ with matrices $\overline{A}$ and $\overline{B}$. Here if the orientation of $e_1$ is coherent to the orientation of $l$,  $\overline{A}=A$; otherwise, $\overline{A}=A^{-1}$. $\overline{B}$ shares  a similar definition. $\overline{A}$ and $\overline{B}$ are called \emph{the matrices provided by $l$ on $e_1$ and $e_2$}.

\begin{definition} \label{neatsimpleWLG}
A simple WLG $G$ is called a \emph{neat simple WLG} if it satisfies one of the following two conditions.
\begin{enumerate}
  \item All  saddle vertices are indexed by $V4$. Moreover, the vertices satisfy the following rules.
  \begin{enumerate}
    \item The gluing matrix between two saddle vertices  is either
                                                                                 $\left(
                                                                                     \begin{array}{cc}
                                                                                       1 & 0 \\
                                                                                       0 & 1 \\
                                                                                     \end{array}
                                                                                   \right)$
                                                                                 or
                                                                                 $\left(
                                                                                   \begin{array}{cc}
                                                                                     -1 & 0 \\
                                                                                     0 & -1\\
                                                                                   \end{array}
                                                                                 \right)$.
    \item Except for two edges, the gluing matrix between a saddle vertex and an $A$  (or a $R$) vertex is either
                                                                                 $\left(
                                                                                     \begin{array}{cc}
                                                                                       1 & 0 \\
                                                                                       k & 1 \\
                                                                                     \end{array}
                                                                                   \right)$
                                                                                 or
                                                                                 $\left(
                                                                                   \begin{array}{cc}
                                                                                     -1 & 0 \\
                                                                                     k & -1\\
                                                                                   \end{array}
                                                                                 \right)$ for some $k\in \mathbb{Z}$.
     \item For the two exceptional edges $e_1$ and $e_2$ with matrices $A$ and $B$, they are adjacent to two $A$ or $R$ vertices $V^1$ and $V^2$. Suppose $\overline{A}$ and $\overline{B}$ are the matrices provided by the unique oriented path starting at $V^1$ and terminating at $V^2$. Then $ \overline{A}\overline{B} = \left(
                                                                                     \begin{array}{cc}
                                                                                       0 & -1 \\
                                                                                       1 & 0 \\
                                                                                     \end{array}
                                                                                   \right)$
                                                                                    or
                                                                                 $\left(
                                                                                   \begin{array}{cc}
                                                                                     0 & 1 \\
                                                                                     -1 & 0\\
                                                                                   \end{array}
                                                                                 \right)$. Suppose $\overline{B}= \left(
                                                                                     \begin{array}{cc}
                                                                                       p & q \\
                                                                                       r & s \\
                                                                                     \end{array}
                                                                                   \right) \in PSL(2,\mathbb{Z})$.
  \end{enumerate}

  \item One saddle vertex is indexed by $V5$ and the others are indexed by $V4$. The matrix between two saddle vertices is the same to the first case.
  The matrix between a saddle vertex and an $A$ (or a $R$) vertex satisfies the following rules.
  \begin{enumerate}
    \item Except for one edge, the gluing matrix between a saddle vertex and an $A$ (or a $R$) vertex is either
                                                                                 $\left(
                                                                                     \begin{array}{cc}
                                                                                       1 & 0 \\
                                                                                       k & 1 \\
                                                                                     \end{array}
                                                                                   \right)$
                                                                                 or
                                                                                 $\left(
                                                                                   \begin{array}{cc}
                                                                                     -1 & 0 \\
                                                                                     k & -1\\
                                                                                     \end{array}
                                                                                 \right)$
                                                                                 for some $k\in \mathbb{Z}$.
    \item For the exceptional edge which starts at an $R$ vertex or terminates at a $A$ vertex, the corresponding matrix $B$ is either
         $\pm\left(
                                                                                     \begin{array}{cc}
                                                                                       1 & -2 \\
                                                                                       1-t & 2t-1 \\
                                                                                     \end{array}
                                                                                   \right)$ (if it starts at an $R$ vertex)
                                      or $\pm \left(
                                                                                     \begin{array}{cc}
                                                                                       2t-1 & 2 \\
                                                                                       t-1 & 1 \\
                                                                                     \end{array}
                                                                                   \right)$ (if it terminates at a $A$ vertex)
         for some $t\in \mathbb{Z}$.
  \end{enumerate}

\end{enumerate}
 \end{definition}

 \begin{definition} \label{neatWLG}
A WLG $G$ is called a \emph{neat WLG} if each simple WLG factor $G_0$ of $G$ is a neat simple WLG. Obviously, a neat simple WLG is a neat WLG.
 \end{definition}

 \begin{definition}\label{signvet}
 Let $G$ be a neat simple WLG. If there exists an $A$ vertex adjacent to an exceptional edge of $G$, then we define the vertex by $V_0$.
 Otherwise, we define $V_0$ to be an $R$ vertex adjacent to an exceptional edge of $G$. Suppose  $V_0'$ is the vertex adjacent to $V_0$
 and $e$ is the edge with vertices $V_0$ and $V_0'$. Assume that $V$ is a vertex in $G$ different to $V_0$.
 Suppose that $k$ is the number of the edges with matrices whose traces are $-2$ in the path between $V$ and $V_0'$.
 Then we can define a kind of
 signature of $V$ associated to $V_0$, $sign (V,V_0)=(-1)^m$  as follows.
 \begin{enumerate}
   \item If $G$ doesn't have a $V5$ vertex, then $m$ can be defined as follows.

   \begin{enumerate}
     \item  If  $V$ is the other vertex adjacent to an exceptional edge in $G$, then $m=k+1$ if  $\overline{A}\overline{B}  = \left(
                                                                                     \begin{array}{cc}
                                                                                       0 & -1 \\
                                                                                       1 & 0 \\
                                                                                     \end{array}
                                                                                   \right)$; $m=k$ if $\overline{A}\overline{B}  = \left( \begin{array}{cc}
                                                                                       0 & 1 \\
                                                                                       -1 & 0 \\
                                                                                     \end{array}
                                                                                     \right)$.
     \item Otherwise, $m=k$.
   \end{enumerate}

   \item If $G$ admits a $V5$ vertex, then $m$ can be defined as follows.

     \begin{enumerate}
       \item  If $B$ is either
         $\left(
                                                                                     \begin{array}{cc}
                                                                                       1 & -2 \\
                                                                                       1-t & 2t-1 \\
                                                                                     \end{array}
                                                                                   \right)$ (when $V_0$ is an $R$ vertex)
                                      or $\left(
                                                                                     \begin{array}{cc}
                                                                                       2t-1 & 2 \\
                                                                                       t-1 & 1 \\
                                                                                     \end{array}
                                                                                   \right)$ (when $V_0$ is an $A$ vertex)
         for some $t\in \mathbb{Z}$, then $m=k$.

         \item If $B$ is either
         $-\left(
                                                                                     \begin{array}{cc}
                                                                                       1 & -2 \\
                                                                                       1-t & 2t-1 \\
                                                                                     \end{array}
                                                                                   \right)$ (when $V_0$ is an $R$ vertex)
                                      or $-\left(
                                                                                     \begin{array}{cc}
                                                                                       2t-1 & 2 \\
                                                                                       t-1 & 1 \\
                                                                                     \end{array}
                                                                                   \right)$ (when $V_0$ is an $A$ vertex)
         for some $t\in \mathbb{Z}$, $m=k+1$.
         \end{enumerate}
 \end{enumerate}
 \end{definition}

 \begin{figure}[htp]
\begin{center}
  \includegraphics[totalheight=7cm]{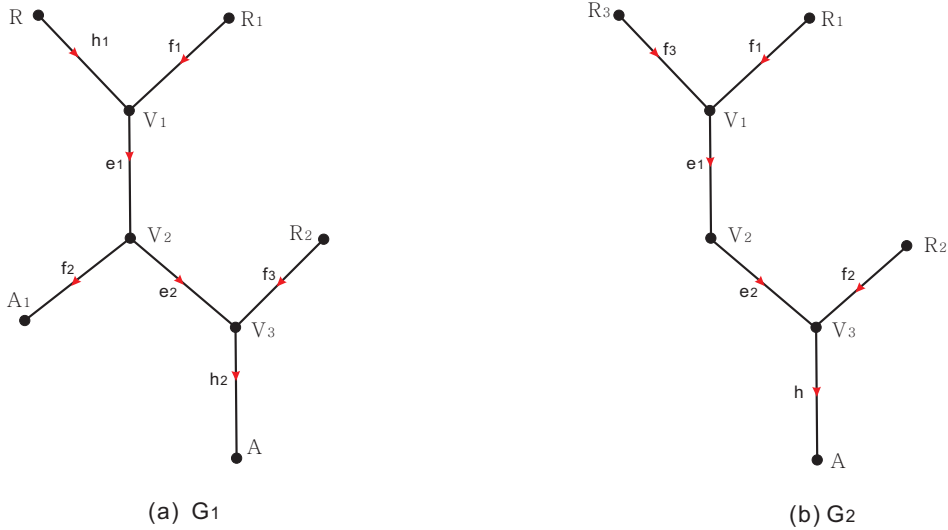}\\
  \caption{two neat simple WLG $G_1$ and $G_2$}\label{WLG}
  \end{center}
\end{figure}

 \begin{example} \label{exampleWLG}
 We introduce two neat simple WLG $G_1$ and $G_2$ as Figure \ref{WLG} shows. In the figure, $R_i$, $V_j$ and $A_k$ are corresponding to $R$, saddle and
 $A$ vertices respectively. In particular, $V_2$ in $G_2$ represents a twisted saddle periodic orbit of a flow. Moreover, $e_i$, $f_j$ and $h_k$ are associated
 to the edges of the WLG. For an edge $e$, we denote the associated gluing matrix by $B_e$.
 \begin{enumerate}
   \item First we define the gluing matrices in $G_1$:

         $$B_{e_1}=\left(
                                                                                     \begin{array}{cc}
                                                                                       1 & 0 \\
                                                                                      0 & 1 \\
                                                                                     \end{array}
                                                                                   \right),
         B_{e_2}=\left(
                                                                                     \begin{array}{cc}
                                                                                       -1 & 0 \\
                                                                                      0 & -1 \\
                                                                                     \end{array}
                                                                                   \right);$$

         $$B_{f_1}=\left(
                                                                                     \begin{array}{cc}
                                                                                       -1 & 0 \\
                                                                                      1 & -1 \\
                                                                                     \end{array}
                                                                                   \right),
         B_{f_2}=\left(
                                                                                     \begin{array}{cc}
                                                                                       1 & 0 \\
                                                                                      -2 & 1 \\
                                                                                     \end{array}
                                                                                   \right),
         B_{f_3}=\left(
                                                                                     \begin{array}{cc}
                                                                                       -1 & 0 \\
                                                                                      -1 & -1 \\
                                                                                     \end{array}
                                                                                   \right);$$

          $$B_{h_1}=\left(
                                                                                     \begin{array}{cc}
                                                                                       -1 & 3 \\
                                                                                      -1 & 2 \\
                                                                                     \end{array}
                                                                                   \right),
          B_{h_2}=\left(
                                                                                     \begin{array}{cc}
                                                                                       3 & 2 \\
                                                                                      1 & 1 \\
                                                                                     \end{array}
                                                                                   \right).$$

     By Definition \ref{signvet}, for the vertex $A$, we can obtain the signature of every vertex of $G_1$ as follows:
     $Sign(V_1, A)=-1$, $Sign(V_2, A)=-1$, $Sign(V_3, A)=1$, $Sign(R_1, A)=1$, $Sign(R_2, A)=1$, $Sign(A_1, A)=1$ and $Sign(R, A)=-1$.

   \item Now we define the gluing matrices in $G_1$:

         $$B_{e_1}=\left(
                                                                                     \begin{array}{cc}
                                                                                       1 & 0 \\
                                                                                      0 & 1 \\
                                                                                     \end{array}
                                                                                   \right),
         B_{e_2}=\left(
                                                                                     \begin{array}{cc}
                                                                                       -1 & 0 \\
                                                                                      0 & -1 \\
                                                                                     \end{array}
                                                                                   \right);$$

         $$B_{f_1}=\left(
                                                                                     \begin{array}{cc}
                                                                                       -1 & 0 \\
                                                                                      1 & -1 \\
                                                                                     \end{array}
                                                                                   \right),
         B_{f_2}=\left(
                                                                                     \begin{array}{cc}
                                                                                       1 & 0 \\
                                                                                      -2 & 1 \\
                                                                                     \end{array}
                                                                                   \right),
         B_{f_3}=\left(
                                                                                     \begin{array}{cc}
                                                                                       -1 & 0 \\
                                                                                      -1 & -1 \\
                                                                                     \end{array}
                                                                                   \right);$$

          $$B_{h}=\left(
                                                                                     \begin{array}{cc}
                                                                                       5 & 2 \\
                                                                                      2 & 1 \\
                                                                                     \end{array}
                                                                                   \right).$$

     By Definition \ref{signvet}, for the vertex $A$, we can obtain the signature of every vertex of $G_2$ as follows:
     $Sign(V_1, A)=-1$, $Sign(V_2, A)=-1$, $Sign(V_3, A)=1$, $Sign(R_1, A)=1$, $Sign(R_2, A)=1$ and $Sign(R_3, A)=1$.
 \end{enumerate}
 \end{example}

\subsection{Neat WLG and Depth $0$ NMS flows on $S^3$}

This subsection focuses on the following question:
How to use (neat) WLG to list depth $0$ NMS flows on $S^3$? 

\begin{lemma}\label{matrixN0}
Let $\phi_t$ be a depth $0$ NMS flow on $S^3$ with a simple WLG. Then there exists a simple WLG $G_1$ of $\phi_t$ such that
each gluing matrix between two saddle vertices is either  $\left(
                                                                                     \begin{array}{cc}
                                                                                       1 & 0 \\
                                                                                       0 & 1 \\
                                                                                     \end{array}
                                                                                   \right)$
                                                                                 or
                                                                                 $\left(
                                                                                   \begin{array}{cc}
                                                                                     -1 & 0 \\
                                                                                     0 & -1\\
                                                                                   \end{array}
                                                                                 \right)$.
\end{lemma}
\begin{proof}
 By $4$ and $5$ of Proposition \ref{filnghd}, when $N$ is the 4th or 5th filtrating neighborhood, each longitude $l_i$ of $N$ is isotopic to the regular
  fibers of $N$ ($N$ is a Seifert manifold with an unique Seifert structure).

  First, let's consider  the local case.  If $N_1$ and $N_2$ are two  adjacent (along a torus $T$) filtrating neighborhoods
  such that $N_i$ is either   4th or 5th filtrating neighborhood in the flow $\phi_t$.
   Without loss of generality, we can suppose $\phi_t$ is transverse outward to $N_1$ and  inward to $N_2$ along $T$.
Let $l^i$ be the longitude of $N_i$ ($i\in\{1,2\}$) on $T$ and $f$ be the gluing homeomorphism from  $N_1$ to  $N_2$ along $T$.
   Since $\phi_t$ is a depth $0$ NMS flow, $f(l^1)\cap l^2 =\emptyset$. By choosing suitable
    meridian $m_2$ in $T$ ($T \subset N_2$), the gluing matrix can be either
                                                                                  $\left(
                                                                                     \begin{array}{cc}
                                                                                       1 & 0 \\
                                                                                       0 & 1 \\
                                                                                     \end{array}
                                                                                   \right)$
                                                                                 or
                                                                                 $\left(
                                                                                   \begin{array}{cc}
                                                                                     -1 & 0 \\
                                                                                     0 & -1\\
                                                                                   \end{array}
                                                                                 \right)$.

 Now let's turn to the global case. To prove the lemma, we only need to prove the claim: by choosing suitable meridians  in
all saddle filtrating neighborhoods,  the gluing matrix between two adjacent filtrating neighborhoods in $N_0$ is either
                                                                                 $\left(
                                                                                     \begin{array}{cc}
                                                                                       1 & 0 \\
                                                                                       0 & 1 \\
                                                                                     \end{array}
                                                                                   \right)$
                                                                                 or
                                                                                 $\left(
                                                                                   \begin{array}{cc}
                                                                                     -1 & 0 \\
                                                                                     0 & -1\\
                                                                                   \end{array}
                                                                                 \right)$.
This claim is followed by choosing suitable  meridians for each saddle filtrating neighborhood  $N$ in $N_0$ as follows.
\begin{itemize}
  \item For each saddle filtrating neighborhood  $N$ in $N_0$ and  an entrance boundary component $T^0$ of $N$:
  \begin{itemize}
    \item if $T^0\subset \partial{N_0}$,  the meridian on $T^0$ can be an arbitrary simple closed curve whose geometrical intersection number with the corresponding longitude is $1$;
    \item otherwise, $N$ is adjacent to another filtrating neighborhood $N'$ in $N_0$ along $T^0$ and the meridian in $T^0$ for $N$ is forced by $N'$ as the rule which we introduced in the local case.

  \end{itemize}
 \item For $\partial^{out}N$, the meridians on $\partial^{out}N$ is forced by the the meridians on  $\partial^{in} N$ under the rules introduced by the
 definitions of coordinates (see Section \ref{coordinate}).
\end{itemize}
\end{proof}

\begin{lemma}\label{notwoV5}
If $G$ is a simple WLG with at least two $V5$ vertices, then $G$ can't be associated to a depth $0$ NMS flow on $S^3$.
\end{lemma}
\begin{proof}
Actually this is a quick consequence of
Lemma \ref{top}. More precisely, if $G$ admits at least two $V5$ vertices, the corresponding saddle simple piece $N_0$ associated to $G$ is homeomorphic to
$M(0,n; \frac{1}{2},\dots, \frac{1}{2})$ where there are $m (m\geq2)$ singular fibers labeled by $\frac{1}{2}$. Then there exists an embedded Klein bottle in $N_0$, but a Klein bottle can't be
embedded into $S^3$.
Therefore, $G$ can't be associated to a depth $0$ NMS flow on $S^3$.
\end{proof}

The following corollary is a direct consequence of Lemma \ref{matrixN0} and Lemma \ref{notwoV5}.

\begin{corollary}\label{topN0}
 Let $\phi_t$ be a depth $0$ NMS flow on $S^3$ with a simple WLG and $N_0$ is a saddle  simple piece of $\phi_t$.
 \begin{enumerate}
   \item  If there doesn't exist a twisted periodic orbit in $\phi_t$, then $N_0$ is homeomorphic to $F_n \times S^1$ ($n\geq 3$) where
    $F_n$ is an $n$  punctured 2-sphere.
  Moreover, the periodic orbits and all the longitude of  $\partial N_0$ can be regarded as regular fibers of $N_0$.
   \item  If there  exists a twisted periodic orbit in $\phi_t$, then $N_0$ is homeomorphic to $\Sigma(0,n; \frac{1}{2})$ ($n\geq 2$).  Moreover, we have:
   \begin{enumerate}
     \item the twisted periodic orbit is the singular fiber of $N_0$;
     \item the periodic orbits and all the longitudes of  $\partial N_0$ can be regarded
  as regular fibers of $N_0$.
   \end{enumerate}
  \end{enumerate}
\end{corollary}

The following two theorems (Theorem \ref{simpleWLG} and Theorem \ref{generalWLG}) explain very well about the relationship between depth $0$ NMS flows and neat WLG. Actually, they
 provide some global descriptions to use neat WLG to list  Depth $0$ NMS flows on $S^3$.

\begin{theorem}\label{simpleWLG}
A depth $0$ NMS flow $\phi_t$ on $S^3$ with a simple WLG always admits a neat simple WLG $G$.
Conversely, for a given neat simple WLG $G$,  there exists a depth $0$ NMS flow on $S^3$ with WLG $G$.
\end{theorem}

Theorem \ref{simpleWLG} is a consequence of the following four lemmas (Lemma \ref{simpleWLG1}, Lemma \ref{simpleWLG2}, Lemma \ref{simpleWLG3} and
Lemma \ref{simpleWLG4}). Notice that there exists a
 hidden principle:   if we fix the coordinates of each filtrating neighborhood, then the gluing matrix between two filtrating neighborhoods is fixed.

\begin{lemma}\label{simpleWLG1}
Let  $\phi_t$ be a depth $0$ NMS flow  on $S^3$ with a simple WLG  and without any twisted saddle periodic orbits,  then $\phi_t$
  can be associated to a neat simple WLG $G$  without any $V5$ vertex (the first kind of neat simple WLG in  Definition \ref{neatsimpleWLG}).
\end{lemma}

\begin{proof}
First of all, by Lemma \ref{matrixN0}, we can construct a simple WLG $G_1$ which satisfies (a) of $1$ in  Definition \ref{neatsimpleWLG}.
Suppose $V_1\sqcup \dots \sqcup V_n= \overline{S^3 - N_0}$ where each $V_i$ is a tubular neighborhood of an attractor or a repeller.
We fix the coordinates of the filtrating neighborhoods in $N_0$ which are associated to $G_1$.

  If there exists one of $\{V_1, \dots, V_n\}$, for instance, $V_1$ such that $V_1$ is glued to $N_0$ sending a meridian of $\partial V_1$ to a
   regular fiber of $N_0$, by Corollary \ref{topN0}, $N_0 \cup V_1$ is homeomorphic to the connected sum of $n-1$ solid tori. Moreover, since
   $(N_0 \cup V_1)\cup V_2 \dots \cup V_n$ is homeomorphic to $S^3$, the gluing map between $V_i$ ($i\in\{2,\dots,n\}$) and $N_0$
   should preserving fibers. It is not difficult to choose  suitable longitude for the longitude of every $V_i$ ($i\in \{1,\dots,n\}$)  such that the new
    WLG $G$ satisfies (b) and (c) of $1$ in Definition \ref{neatsimpleWLG}. Therefore, $G$ is a
    neat simple WLG.  Notice that in
   this case, the two exceptional edges admit $V_1$ and $V_2$ as two ends respectively.

  Otherwise, $N_0 \cup V_1\cup \dots \cup V_n$ endows $S^3$ a Seifert structure.  Lemma \ref{top} tells us that a Seifert structure
   of $S^3$ at most contains two singular fibers. Without loss of generality, suppose $V_1$ and $V_2$ contain all the singular fibers.
   Naturally, the two exceptional edges admit $V_1$ and $V_2$ as two ends respectively. By choosing suitable longitudes to $V_3,\dots, V_n$, the gluing map
   from $V_i$ ($i\geq 3$) to $N_0$ preserving fibers. Then  the gluing matrices of the corresponding WLG $G$ satisfy (b) of $1$ in
    Definition \ref{neatsimpleWLG}.  Since $G$ satisfies (a) of $1$ in Definition \ref{neatsimpleWLG} and Corollary \ref{topN0},
    $N_0 \cup V_3 \cup \dots \cup V_n$ is
   homeomorphic to $T^2 \times [0,1]$, one can naturally check that $G$ satisfies (c) of $1$ in Definition \ref{neatsimpleWLG}.
\end{proof}

\begin{lemma}\label{simpleWLG2}
If $G$ is a neat simple WLG without any $V5$ vertex, then there exits a depth $0$ NMS flow $\phi_t$ on $S^3$ with WLG $G$.
\end{lemma}

\begin{proof}
Under the assumption of (a) of $1$ in Definition \ref{neatsimpleWLG}, one can easily build saddle simple piece $N_0$ associated to $G$  such
 that $N_0$ satisfies the following conditions.
 \begin{itemize}
   \item $N_0$ is homeomorphic to $F_n \times S^1$ ($n\geq 3$) where $F_n$ is a $n$ punctured 2-sphere.
   \item There doesn't exist a heteroclinic trajectory connecting saddle orbits in $N_0$.
   \item The coordinates restricted to $\partial N_0$ is standard in the following sense. Every longitude  is parallel to a fiber of $N_0$
   and all meridians  restricted to $\partial N_0$ bound a surface $\Sigma$ transverse to the fibers in $N_0$.
 \end{itemize}

Then one can glue $N_0$ to  $n-2$  filtrating neighborhoods of attractors and repellers $V_3, \dots, V_n$ to a new compact 3-manifold $N_1$ with flow
 such that the gluing matrices between
$V_3 \cup \dots \cup V_n$ and $N_0$ satisfy (b) of $1$ in Definition \ref{neatsimpleWLG}.
 Notice the position of the coordinates restricted to $\partial N_0$ and the matrices shown in (b) of $1$ in Definition \ref{neatsimpleWLG},
 $N_1$ is homeomorphic to $T^2 \times [0,1]$.

In the end, we glue $N_1$ to  two   filtrating neighborhoods   associated to the exceptional attractor and repeller vertices
such that the gluing matrices satisfy (c) of $1$ in Definition \ref{neatsimpleWLG}. Then we obtain a closed three manifold $M$ with flow $\phi_t$.
 Notice that $N_0 \cong T^2 \times [0,1]$ and $G$ satisfies (a) and (c) of $1$ in Definition \ref{neatsimpleWLG},
 one can easily check that $M$ is homeomorphic to $S^3$ and $\phi_t$ is a depth $0$ NMS flow with WLG $G$.
\end{proof}

\begin{lemma} \label{simpleWLG3}
Let $\phi_t$ be a depth $0$ NMS flow  on $S^3$ with twisted saddle periodic orbits and with a simple WLG.
Then  there exists a neat simple WLG $G$ with a $V5$ vertex (the second kind of neat simple WLG in  Definition \ref{neatsimpleWLG}) of $\phi_t$.
\end{lemma}

\begin{proof}
 First of all, by Lemma \ref{notwoV5}, there is exactly one twisted saddle periodic orbit in $\phi_t$.

Similar to case $1$, by Lemma \ref{matrixN0}, we can construct a simple WLG $G_1$ of $\phi_t$ which satisfies (a) of $1$ in
Definition \ref{neatsimpleWLG}. Moreover, for a saddle simple piece $N_0$, $N_0$ can be decomposed to several filtrating neighborhoods with coordinates
associated to $G_1$. We fix the coordinates of the filtrating neighborhoods in $N_0$ which are associated to $G_1$. Assume that
 $V_1\sqcup \dots \sqcup V_n= \overline{S^3 - N_0}$ where each $V_i$ is a tubular neighborhood of an attractor or a repeller.

If a meridian of some $V_i$ is glued to a regular fiber of $N_0$. Then one can find an embedded $\mathbb{R}P^2$ in $V_i \cup N_0$. But
 $\mathbb{R}P^2$ can't be embedded into $S^3$.

 Otherwise, $N_0 \cup V_1\cup \dots \cup V_n$ endows $S^3$ a Seifert structure in our discussion.
 By  Lemma \ref{top} and the fact $N_0$ is   homeomorphic to $\Sigma(0,n; \frac{1}{2})$ ($n\geq 2$) (Corollary \ref{topN0}),
 the Seifert structure of $S^3$ provided by
  $N_0 \cup V_1\cup \dots \cup V_n$ is $\Sigma(0,0; \frac{1}{2}, \frac{\beta}{\alpha})$. Then one of $V_1, \dots, V_n$, for instance $V_1$,
   provides the other singular fiber of the Seifert structure. Naturally, the exceptional edge is associated to $V_1$.
   By changing longitudes of $V_2,\dots, V_n$ suitably, the gluing map from $V_i$ ($i\geq 3$) to $N_0$ preserving fibers and the gluing matrices
   satisfy (a) of 2 in Definition \ref{neatsimpleWLG}. By the way, it is easy to show that $W= N_0 \cup V_2 \cup \dots \cup V_n$
    is homeomorphic to $\Sigma(0,1; \frac{1}{2})$.

    If $V_1$ is a $A$ vertex,  to ensure that $W\cup V_1$ is homeomorphic to $S^3$,
    after choosing a suitable coordinate for $V_1$,
    the gluing matrix $B$ of the exceptional edge is either $\left(
                                                                                     \begin{array}{cc}
                                                                                       2t-1 & 2 \\
                                                                                       t-1 & 1 \\
                                                                                     \end{array}
                                                                                   \right)$
   or $- \left(
                                                                                     \begin{array}{cc}
                                                                                       2t-1 & 2 \\
                                                                                       t-1 & 1 \\
                                                                                     \end{array}
                                                                                   \right)$ for some $t\in \mathbb{Z}$.
   If  $V_1$ is an $R$ vertex, similar to the case when $V_1$ is a $A$ vertex,  to ensure that $W\cup V_1$ is homeomorphic to $S^3$,
    after choosing a suitable coordinate for $V_1$,
    the gluing matrix $B$ of the exceptional edge is either $\left(
                                                                                     \begin{array}{cc}
                                                                                      1 & -2 \\
                                                                                       1-t & 2t-1 \\
                                                                                     \end{array}
                                                                                   \right)$
   or $- \left(
                                                                                     \begin{array}{cc}
                                                                                     1 & -2 \\
                                                                                       1-t & 2t-1 \\
                                                                                     \end{array}
                                                                                   \right)$ for some $t\in \mathbb{Z}$.
     Therefore, the gluing matrix  from $V_1$  to $W$ associated to the coordinates as above
   satisfies (b) of 2 in Definition \ref{neatsimpleWLG}.

   In summary, the WLG associated to the coordinates as above is a neat simple WLG  with $V5$ vertex (Definition  \ref{neatsimpleWLG}).
\end{proof}

\begin{lemma} \label{simpleWLG4}
If $G$ is a neat simple WLG with a $V5$ vertex, then there exists a depth $0$ NMS flow $\phi_t$ on $S^3$ with WLG $G$.
\end{lemma}

\begin{proof}
Under the assumption of the  matrices between saddle vertices, one can easily build a saddle simple piece $N_0$ associated to $G$  such that $N_0$
   satisfies the following conditions.
 \begin{itemize}
   \item $N_0$ is homeomorphic to $\Sigma(0,n; \frac{1}{2})$ ($n\geq 2$).
   \item There doesn't exist a heteroclinic trajectory connecting saddle orbits in $N_0$.
   \item Each longitude in $\partial N_0$  is parallel to a fiber of $N_0$.
 \end{itemize}

Then one can glue $N_0$ to  $n-1$ corresponding filtrating neighborhoods of attractors and repellers $V_2, \dots, V_n$ to a new compact 3-manifold
 $N_1$ with flow such that the gluing matrices between
$V_2 \cup \dots \cup V_n$ and $N_0$ satisfy (a) of $2$ in Definition \ref{neatsimpleWLG}.
Notice the position of the coordinates restricted to $\partial N_0$ and the matrices shown in (a) of $2$ in Definition \ref{neatsimpleWLG},
 $N_1$ is homeomorphic to a solid torus.

In the end, we glue $N_1$ to  a   filtrating neighborhood associated to the exceptional attractor or
repeller vertex such that the gluing matrices satisfy (b) of $2$ in Definition \ref{neatsimpleWLG}.
Then we obtain a closed three manifold $M$ with flow $\phi_t$.
Under these gluing conditions, one can easily check that $M$ is homeomorphic to $S^3$ and $\phi_t$ is a depth $0$ NMS flow with WLG $G$.
\end{proof}

\begin{theorem}\label{generalWLG}
A neat WLG $G$ is associated to a depth $0$ NMS flow on $S^3$ if and only if they satisfy the following conditions:
\begin{enumerate}
  \item every simple WLG factor admits at most one special vertex;
  \item if there exists a special vertex in a simple WLG factor $G_0$ of $G$, then there doesn't exist a $V5$ vertex in $G_0$.
\end{enumerate}
Conversely, a depth $0$ NMS flow $\phi_t$  on $S^3$ always admits a neat WLG $G$ satisfying the above two conditions.
\end{theorem}

\begin{proof}
(\emph{the necessity of the first part}) Let $\phi_t$ be a depth $0$ NMS flow with a neat WLG $G$.
 Assume that a simple WLG  factor $G_1$ of $G$ admits two special vertices. Let $N_0$ be a saddle simple piece of $\phi_t$ associated to $G_1$.
  By Corollary \ref{topN0}, $N_0$  is
       either homeomorphic to $F_n \times S^1$ ($n\geq 3$) or $\Sigma (0,n; \frac{1}{2})$ ($n\geq 2$). The fact that $G_1$ admits two special
       vertices means that there are two connected components of $\partial N_0$ are glued to two solid tori in such a way: two regular
        fibers (in these two connected
       components respectively) are glued to two meridians of the two solid tori respectively. It is easy to observe that there exists a non-separating $S^2$ inside.
       This conflicts with the fact that $S^3$ doesn't contain a non-separating $S^2$.  Therefore, $G$ satisfies the first condition of the theorem.

      If  $G$ doesn't satisfy the second condition of the Theorem, then there exists a simple WLG factor $G_1$ of $G$
      which admits a special  vertex and a $V5$ vertex. Let $N_0$ be a saddle simple piece of $\phi_t$ associated to $G_1$.
     By Corollary \ref{topN0} and the fact that $G_1$ admits a $V5$ vertex, $N_0$
      is homeomorphic to  $\Sigma (0,n; \frac{1}{2})$ ($n\geq 2$).  The fact that $G_1$ admits a special vertex means that there is a connected
      component of $\partial N_0$ which is glued to a solid torus such that  there exists a regular fiber of $N_0$ which is glued to a meridian
       of the solid torus.
      Then one can  observe that there exists an embedded surface in $S^3$ which is homeomorphic to $\mathbb{R}P^2$. This is obviously impossible.
       Therefore, $G$ satisfies the second condition of the theorem.

(\emph{the sufficiency of the first part})
Let $G$ be a neat WLG which satisfies the two conditions in the theorem.
Suppose $\{G_1, \dots, G_n\}$ is the simple WLG decomposition of $G$.  For every $G_i$ ($i\in \{1,\dots, n\}$),
by Theorem \ref{simpleWLG}, there exists a depth $0$ NMS flow $\phi_t^i$  associated to $G_i$.
Following the relationship between  $G$ and $\{G_1, \dots, G_n\}$, one can naturally build NMS flow $\phi_t$ on some
 closed orientable 3-manifold $M$ with WLG $G$ by doing the converse of flow splits among
$\phi_t^1, \dots, \phi_t^n$. The two conditions in the theorem ensure that
$M$ is homeomorphic to $S^3$. Moreover, if we choose the gluing maps carefully (one can come back to Section \ref{split} and Proposition
\ref{filnghd} to check the detail), there doesn't exist a heteroclinic trajectory connecting two
 saddle orbits in $\phi_t$. Therefore, we have constructed a depth $0$ NMS flow $\phi_t$ with WLG $G$.

(\emph{proof of the second part})
  Let $G'$ be a WLG of $\phi_t$ with the simple WLG decomposition $\{G_1', \dots, G_n'\}$. The NMS flows set $\{\phi_t^1,\dots,\phi_t^n\}$
   is the flow splitting of $\phi_t$ associated to $G$. Here $G_i'$ ($i\in \{1,\dots, n\}$) is a WLG of $\phi_t^i$. By Theorem \ref{simpleWLG},
   there exist some coordinates for the filtrating neighborhoods of $\phi_t^i$ such that the WLG associated to the coordinates is a neat simple
   WLG $G_i$. Now we provide a new coordinate to  every connected component of every filtrating neighborhood of $\phi_t$ associated to $G'$.

   Suppose $W$ is a solid torus component  in the prime decomposition of a filtrating
   neighborhood $N$ in $\phi_t$. By Section \ref{coordinate} and Section \ref{split},
   we can suppose that $W$ is  associated to a special vertex $V_1$ in $G_1'$ such that the filtrating neighborhood $N_1$ labeled by $V_1$
   comes from $W$. Naturally, the maximal invariant set of
   $N_1$ is either an attractor or a repeller.   Then the coordinate in $\partial N_1$ naturally provides a coordinate
   of $\partial W \subset \partial N$.

   For another boundary component $\Sigma$ of some filtrating neighborhood of $\phi_t$ associated to $G'$, $\Sigma$ is  a boundary component
   of some filtrating neighborhood of $\phi_t^i$ associated to $G_i$ for some $i\in \{1,\dots, n\}$.

   Now one can easily  check that the WLG $G$  provided by  the new coordinates is a neat WLG.
\end{proof}

\section{Indexed links of Depth $0$ NMS flows on $S^3$}\label{link}
The periodic orbits of an NMS flow on $S^3$ form a \emph{indexed link}, that is, an oriented link with the index $0,1$ or $2$ attached to each
component.  The orientations of the indexed link are naturally endowed by the orientation of the flow. The index of a component of a periodic orbit $\gamma$ is defined as the dimension of the strong stable manifold of $\gamma$. Wada \cite{Wa}  characterized the set of indexed links which arise as the closed orbits of an NMS flow on $S^{3}$ in terms of a generator and six operations. In this section, based on the work in the last section, we will give some direct descriptions about indexed links of Depth $0$ NMS flows on $S^3$.

The following proposition implies that WLG decides the indexed link of NMS flows. The proposition  is a direct consequence of the definition of WLG (Section \ref{coordinate}). Therefore, we omit its proof here.

\begin{proposition} \label{decidelink}
Let $G$ be a WLG which can be associated to an NMS flow on $S^3$. Suppose $\phi_t^1$ and $\phi_t^2$ are two NMS flows on $S^3$, then  the indexed links of $\phi_t^1$ and $\phi_t^2$ are the same.  The indexed link of $\phi_t^1$ is called the indexed link of $G$.
\end{proposition}

Now let's state a proposition about some connections between split surgeries and Wada's operation for all NMS flows on $S^3$. The proof can be easily followed by a routine check about the definition of  flow splits and a similar argument to the proof of case (a), (b) and (c) of the main theorem of Wada \cite{Wa}.

\begin{proposition} \label{Wadaop}
Suppose $\phi_t$ is an NMS flow on $S^3$ with a WLG $G$. Moreover, suppose $\phi_t^1$ and $\phi_t^2$ are the flows obtained  by doing
flow split of $\phi_t$ associated to a $V_i$ ($i\in \{1,2,3\}$) vertex of $G$. Let $L_1,L_2$ and $L$ be the indexed links of $\phi_t^1, \phi_t^2$ and $\phi_t$ correspondingly. Then $L_1,L_2$ and $L$ have the following relations.
\begin{enumerate}
  \item If $i=1$, $L=L_1 \cdot K_{V_1} \cdot L_2$ where $K_{V_1}$ is a normal indexed $1$ trivial knot corresponding to $V_1$ and $L_1 \cdot K_{V_1} \cdot L_2$ is the split sum of $L_1, L_2$ and $K_{V_1}$. This is Wada's operation I.
  \item If $i=2$, $L=L_1 \cdot K_{V_2}\cdot(L_2 -K)$ where $K_{V_2}$ is a normal indexed $1$ trivial knot corresponding to $V_2$ and $K$ is an index $0$ or $2$ knot in $L_2$. This is Wada's operation II.
  \item If $i=3$, $L=(L_1-K_1) \cdot K_{V_3}\cdot(L_2 -K_2)$ where $K_{V_3}$ is a normal indexed $1$ trivial knot corresponding to $V_3$ and, $K_1$ and $K_2$ are  index $0$ and index $2$ knots in $L_1$ and $L_2$ (or $L_2$ and $L_1$) correspondingly. This is Wada's operation III.
\end{enumerate}
\end{proposition}

In Section \ref{WLGD}, we use neat WLG to represent depth $0$ NMS flows on $S^3$ very well. Combine this with Proposition \ref{decidelink},
naturally we have the following question.
 \begin{question}\label{Qindlink}
 For a given neat WLG $G$ which satisfies the two conditions of Theorem \ref{generalWLG}, what is the indexed link of $G$?
 \end{question}
 Proposition \ref{Wadaop} tells us that for a neat WLG $G$, if we know the indexed link of every simple WLG factor of $G$, then we can easily get the indexed link of $G$ through a finite steps of Wada's operations I, II and III. Therefore, to answer Question \ref{Qindlink}, we only need to discuss the indexed link of a neat simple WLG. The following two theorems (Theorem \ref{noV5link} and Theorem \ref{V5link}) just deal with such a question.

\begin{theorem}\label{noV5link}
Let $G$ be a neat simple WLG without a $V5$ vertex and $V_0$ be  a vertex in $G$ defined in Definition \ref{signvet}. Suppose that the number of the
 vertices of $G$ is $n$ and $L$ is the indexed link of $G$. Then $L$ and $G$ satisfy the following conditions.
\begin{enumerate}
    \item $L$ is composed of a Hopf link and $n-2$ $(p,q)$ cable knots (to the knot associated to $V_0$) which are pairwise parallel. Here $(p,q)$ is the first line vector of $\overline{B}$ in Definition \ref{neatsimpleWLG}.
    \item Each one of the Hopf link $K_0 \cup K_1$ is either an attractor or a repeller. Moreover, $K_0$ and $K_1$ are correspondence to two vertices $V_0$ and $V_1$ in $G$ respectively.
    \item Moreover, $L$ satisfy the following orientation conditions.
        \begin{enumerate}
          \item If $Sign(V_1, V_0)=1$, then $K_0 \cup K_1$ is a positive Hopf link; otherwise,  $K_0 \cup K_1$ is a negative Hopf link.
          \item For another vertex $V\in G$, $V$ is correspondence to a $Sign(V, V_0) (p,q)$ cable knot of $K_0$.
        \end{enumerate}
  \end{enumerate}
\end{theorem}

\begin{proof}
By $1$ of Corollary \ref{topN0}, a saddle simple piece $N_0$ is homeomorphic to $F_n \times S^1$ with periodic orbits as some regular fibers. Here $F_n$ is a $n$ punctured 2-sphere ($n\geq 3$). Then one can check that
 $U_0\sqcup U_1 \dots \sqcup V_{n-1}= \overline{S^3 - N_0}$ where each $U_i$ is a tubular neighborhood of an attractor or a repeller. In particular,
 $U_0$ and $U_1$ are associated to $V_0$ and $V_1$ respectively. Assume that $U_2$ is the filtrating neighborhood associated to the vertex adjacent
 to $V_0$.

 Suppose $N=\overline{S^3 - U_0\cup U_1}$.
 By (b) of Definition \ref{neatsimpleWLG}, $N$ is homeomorphic to $T^2 \times [0,1]$ such that every periodic orbit in $N$ is isotopic to the longitudes
 in $\partial U_2$.  Furthermore, by (c) of Definition \ref{neatsimpleWLG}, we can obtain that $L$ satisfies $1$ and $2$ of the theorem.

 For the orientations of $L$, one can follow  Definition \ref{neatsimpleWLG} and Definition \ref{signvet} to routinely check that $L$ satisfy $3$ of the theorem. We omit the detail here.
\end{proof}

Similar to the proof of Theorem \ref{noV5link}, one can use $2$ of Corollary \ref{topN0}, Definition \ref{neatsimpleWLG} and Definition \ref{signvet}
to prove the following theorem. For simplify, we omit the detail here.

\begin{theorem}\label{V5link}
Let $G$ be a neat simple WLG with a $V5$ vertex and $V_0$ be  a vertex in $G$ defined in Definition \ref{signvet}. Suppose that the number of the
 vertices of $G$ is $n$ and $L$ is the indexed link of $G$. Then $L$ and $G$ satisfy the following conditions.
\begin{enumerate}
    \item $L$ is composed of a Hopf link and $n-2$ $(2t-1,2)$ cable knots (to the knot associated to $V_0$) which are pairwise parallel. Here $t$ is defined in $B$ of Definition \ref{neatsimpleWLG}.
    \item The Hopf link $K_0 \cup K_1$  are correspondence to two vertices $V_0$ and $V_1$ in $G$ respectively. Here $V_1$ is the unique $V5$ vertex in $G$.
    \item Moreover, $L$ satisfy the following orientation conditions.
        \begin{enumerate}
          \item If $Sign(V_1, V_0)=1$, then $K_0 \cup K_1$ is a positive Hopf link; otherwise,  $K_0 \cup K_1$ is a negative Hopf link.
          \item If   $V\in G -\{V_0,V_1\}$, $V$ is correspondence to a $Sign(V, V_0) (2t-1,2)$ cable knot of $K_0$.
        \end{enumerate}
  \end{enumerate}
\end{theorem}

For instance, we can read the indexed links of the WLG in Example \ref{exampleWLG} as follows.

\begin{example}\label{examplelink}
\begin{enumerate}
  \item Suppose $L_1= \{K_A, K_{A_1}, K_{V_1}, K_{V_2}, K_{V_3}, K_R, K_{R_1}, K_{R_2}\}$ is the indexed link of $G_1$ in Example \ref{exampleWLG}. Here $K_V$ is the indexed knot associated to $V$.   Then by Theorem \ref{noV5link}, $L_1$ satisfies the following conditions.
      \begin{enumerate}
        \item If we forget the orientation of $L_1$, $K_A \cup K_R$ is a Hopf link and $L-K_A \cup K_R$ is composed of $6$ $(3,2)$
        cable knots of $K_A$.
        \item Now we provide the orientation information. $K_A \cup K_R$ is a negative Hopf link, everyone of $\{K_{A_1}, K_{V_3},
        K_{R_1}, K_{R_2}\}$ is a $(3,2)$ cable knot and everyone of $\{K_{V_1}, K_{V_2} \}$ is a $-(3,2)$ cable knot.
      \end{enumerate}

  \item  Suppose $L_2= \{K_A,  K_{V_1}, K_{V_2}, K_{V_3},  K_{R_1}, K_{R_2}, K_{R_3}\}$ is the indexed link of $G_2$ in Example \ref{exampleWLG}. Here $K_V$ is the indexed knot associated to $V$.   Then by Theorem \ref{V5link}, $L_2$ satisfies the following conditions.
      \begin{enumerate}
        \item If we forget the orientation of $L_1$, $K_A \cup K_{V_2}$ is a Hopf link and $L-K_A \cup K_2$ is composed of $5$ $(5,2)$
        cable knots of $K_A$.
        \item Now we provide the orientation information. $K_A \cup K_{V_2}$ is a negative Hopf link, everyone of $\{K_{V_3}, K_{R_1},
        K_{R_2}, K_{R_3}\}$ is a $(3,2)$ cable knot and $K_{V_1}$ is a $-(5,2)$ cable knot.
      \end{enumerate}
\end{enumerate}
\end{example}

\begin{remark}
Actually, if an indexed link $L$ on $S^3$  satisfies the conditions $1$ and $2$ of Theorem \ref{noV5link} or Theorem \ref{V5link}, then
there exits a depth $0$ NMS flow
on $S^3$ with $L$ as its periodic orbits. This can be proved by a routine construction similar to the proof of Theorem \ref{simpleWLG}.
\end{remark}

By the description of indexed links of depth $0$ NMS flows, we can positively answer Question \ref{simpleflow} for depth $0$ NMS flows.

\begin{corollary} \label{anwq}
Suppose  $\phi_t$ is  a depth $0$ NMS flow on $S^3$ and  can be associated to a simple WLG. Then every WLG of $\phi_t$ is simple.
\end{corollary}
\begin{proof}
By Theorem \ref{noV5link} and Theorem \ref{V5link},  the indexed link of a depth $0$ NMS flow on $S^3$ endowed with a simple WLG doesn't contain a splitable trivial
knot.
On the other hand, by Proposition \ref{Wadaop}, the indexed link of a depth $0$ NMS flow on $S^3$ endowed with a  WLG which isn't simple contains a
 splitable indexed $1$ trivial knot. The corollary is a consequence of these two facts.
\end{proof}

\section{A simplified Umanskii Theorem} \label{Umanskii}
For depth $0$ NMS flows on $S^3$, it is obvious that WLG is quite sensitive to topologically equivalent classes. One naturally expect
that WLG can detect a depth $0$ NMS flow up to topological equivalence. But indeed, this is not true. In particular, it isn't sensitive to
the combinatorial information about the stable and unstable manifolds of the saddle periodic orbits.

In this section, we focus on a criterion whether two depth $0$ NMS flows on $S^3$ are topologically equivalent.
As we have introduced in Section \ref{introduction}, Umanskii \cite{Um} has systematically
discussed such a question. He have constructed some combinatorial invariants. For convenient, we call them  Umanskii invariant.
Actually in the case of depth $0$ NMS flows, the dynamical system  is much simpler than a generical study
of NMS flows. For distinguishing the topologically equivalent classes of   depth $0$ NMS flows,  form our viewpoint, a reduced form of $\omega$-schemes in Umanskii invariant is enough. This reduced form
is called  accompany graph which is
used to describe the relative position of invariant
manifolds of saddle periodic trajectories near attractors and repellers.

\begin{definition}
An abstract \emph{accompany graph} is a tori set $\mathcal{T} =\{T_1,\dots, T_n\}$  with pairwise disjoint oriented simple closed curves $\{c_1, \dots, c_m\}$
 in $T_1 \cup\dots \cup T_n$.
\end{definition}

\begin{definition}
For a depth $0$ NMS flow $\phi_t$, it naturally provides an accompany graph as follows.
Suppose $\mathcal{T}$ is the boundary of a small tubular  neighborhood of  the attractors and the repellers  of $\phi_t$.
Let $c_1 \cup \dots \cup c_m$ be the intersection of the invariant manifold of the saddle periodic orbits  and $\mathcal{T}$. Then $(\mathcal{T},\{c_1, \dots, c_m\})$ is called  \emph{an accompany graph of $\phi_t$}.
\end{definition}

\begin{definition}\label{accequ}
Let $G_1$ and $G_2$ be two WLG of two depth $0$ NMS flows $\phi_t^1$ and $\phi_t^2$ correspondingly. Suppose $\{c_1^1, \dots, c_m^1\}$ and $\{c_1^2, \dots, c_m^2\}$ are the corresponding accompany graphs on $\partial V^1 = T_1^1\cup \dots \cup T_n^1$ and $\partial V^2 =T_1^2\cup \dots \cup T_n^2$ correspondingly. Here $V^1$ and $V^2$ are two small tubular  neighborhoods of  the attractors and the repellers  of $\phi_t^1$ and $\phi_t^2$ correspondingly.  Then we call such two accompany graphs are \emph{equivalent under $G_1$ and $G_2$ by $g$} if $G_1$ and $G_2$ are WLG equivalent under some WLG map $g$ and there exists a homeomorphism $f: \partial V^1 \rightarrow \partial V^2$ satisfies the following conditions.
\begin{enumerate}
  \item $f$ coordinates to $g$, i.e., if $f(T_1^i)=T_2^j$ ($i,j \in\{1,\dots,n\}$) and $\gamma_i^1$ and $\gamma_j^2$ are the periodic orbits associated to $T_1^i$ and $T_2^j$ correspondingly, then $g(\langle \gamma_i^1 \rangle)=\langle\gamma_j^2 \rangle$.  Here $\langle \gamma_i^1 \rangle$ and $\langle\gamma_j^2 \rangle$ are the vertices in $G_1$ and $G_2$ associated to $\gamma_i^1$ and $\gamma_j^2$ correspondingly.
  \item Under the coordinates associated to $G_1$ and $G_2$, $f$ is isotopic to  $\left(
                                                                                     \begin{array}{cc}
                                                                                       1 & 0 \\
                                                                                       0 & 1 \\
                                                                                     \end{array}
                                                                                   \right)$.
  \item For each $i\in \{1,\dots,m\}$, $f(c_i^1)=c_i^2$.
  \item If $c_i^1$ belongs to the stable (unstable)  manifold of a saddle periodic orbit $\gamma^1$ of $\phi_t^1$, then
      $c_i^2$ belongs to the stable (unstable)  manifold of a saddle periodic orbit $\gamma^2$ of $\phi_t^2$. Moreover, $\gamma^1$ and $\gamma^2$ are correspondent under $g$.
\end{enumerate}
\end{definition}

\begin{remark}
Suppose $\{c_1,\dots, c_m\}$ and $\{c_1',\dots, c_m'\}$ are two  accompany graphs of a depth $0$ NMS flow $\phi_t$.  Then we can easily check that they are equivalent. Therefore, in the following of the paper, we think the accompany graphs of $\phi_t$ are unique.
\end{remark}

There exists a natural restriction about $f$ in Definition \ref{accequ}. The restriction is useful in the next section.
First we define a subgroup $\Gamma$ of $SL(2,\mathbb{Z})$.

Let $$\Gamma=\left\langle \left(
            \begin{array}{cc}
              1 & 0 \\
              p & 1 \\
            \end{array}
          \right)
  , \left(
            \begin{array}{cc}
              1 & 0 \\
              q & -1 \\
            \end{array}
          \right)
 \right\rangle,\ (p,q \in \mathbb{Z})$$ be a subgroup of $SL(2,\mathbb{Z})$ generated by $\left(
            \begin{array}{cc}
              1 & 0 \\
              p & 1 \\
            \end{array}
          \right)$
          and $\left(
            \begin{array}{cc}
              1 & 0 \\
              q & -1 \\
            \end{array}
          \right)$.

 \begin{proposition}\label{df}
 Let $\phi_t^1$ and $\phi_t^2$ be  two depth $0$ NMS flows on $S^3$. Suppose $f: S^3 \rightarrow S^3$ is a topologically equivalent map from $\phi_t^1$ to $\phi_t^2$. Moreover, suppose $V_1$ is a filtrating neighborhood of an attractor $\gamma_1$ in $\phi_t^1$. Set $\gamma_2 = f(\gamma_1)$ and $V_2 = f(V_1)$. Under two given coordinates as Section \ref{coordinate} and up to mapping class group, $f|_{\partial V_1} \in \Gamma$.
 \end{proposition}

\begin{proof}
The proposition can be easily followed from the following facts and the definitions of coordinates in Section \ref{coordinate}.
\begin{itemize}
  \item $f(\gamma_1) =\gamma_2$ and $f$ preserves the orientations.
  \item Suppose $c$ is a meridian of $V_1$. Then $f(c)$ is a meridian of $V_2$.
\end{itemize}
\end{proof}

Now we can say that WLG and accompany graph decide the topologically equivalent class of depth $0$ NMS flows on $S^3$. More precisely, we have the following main theorem of this section.

\begin{theorem}\label{topequ}(a simplified Umanskii Theorem)
Let $\phi_t^1$ and $\phi_t^2$ be two depth $0$ NMS flows on $S^3$. Assume that there exist WLG $G_1$ and $G_2$ of $\phi_t^1$ and $\phi_t^2$ correspondingly which satisfy the following conditions:
\begin{enumerate}
  \item $G_1$ and $G_2$ are WLG equivalent under a WLG map $g:G_1 \rightarrow G_2$;
  \item the accompany graphs of $\phi_t^1$ and $\phi_t^2$ are equivalent under $G_1$ and $G_2$ by $g$.
\end{enumerate}
Then $\phi_t^1$ and $\phi_t^2$ are topologically equivalent.
\end{theorem}

As the name of the theorem, this theorem is a simplified theorem of the main theorem of Umanskii's paper \cite{Um} (we call it Umanskii Theorem) which says that
Umanskii invariant decide the topologically equivalent class of Morse-Smale flows on three manifolds. Let's explain more about this. Indeed, WLG carries information
on limit behavior of wandering trajectories and topology of connected components of the
set $S^3$ without closure of invariant stable and unstable manifolds of closed saddle
trajectories. The same information may be deduced from the cells which are part of
Umanskii invariant. Accompany graph  describes relative position of invariant
manifolds of saddle periodic trajectories near sinks and sours, and the same information is
given by $\omega$-scheme in Umanskii invariant. Theorem \ref{topequ} can be regarded as a consequence of Umanskii Theorem under these relationships.   Therefore, we omit the proof of the theorem in the main part of the paper.
The reader also can find a proof of Theorem \ref{topequ} in Appendix \ref{appendix} of the paper.

The following corollary is a consequence of Theorem \ref{topequ}.

\begin{corollary}\label{finite}
Let $G$ be a WLG (in particular, a neat WLG) associated to a depth $0$ NMS flow on $S^3$, then up to topological equivalence, there are finitely many
depth $0$ NMS flows on $S^3$ with $G$ as a WLG.
\end{corollary}

\begin{remark}
Maybe two depth $0$ NMS flows $\phi_t^1$ and $\phi_t^2$ on $S^3$ are topologically equivalent but they don't satisfy the condition in Theorem \ref{topequ}. For instance, two depth $0$ NMS flows $\phi_t^1$ and $\phi_t^2$ with two periodic orbits such that the periodic orbits of $\phi_t^1$ form a positive Hopf link and the periodic orbits of $\phi_t^2$ form a negative Hopf link.  It is not difficult to check that $\phi_t^1$ and $\phi_t^2$ are topologically equivalent but there doesn't exist a WLG of $\phi_t^1$ and $\phi_t^2$ satisfying the condition in Theorem \ref{topequ}.
For more information about these flows, see Section \ref{2per}.
\end{remark}

\section{Depth $0$ NMS flows with periodic orbits number no more than $4$} \label{example}
In this section, we will list and classify all kinds of depth $0$ NMS flows on $S^3$ up to $4$ periodic orbits.  Obviously the periodic orbits number $n$ of a depth $0$ NMS flow with periodic orbits number no more than $4$ has three possibilities: $n=2$, $n=3$ or $n=4$.  Notice that
 in the cases $n=2$ and $n=3$, every NMS flow is a depth $0$ NMS flow. Therefore, actually we will classify NMS flows on $S^3$ with periodic orbits
 number $2$ and $3$.

 We will divide the discussion to
three cases under periodic orbits number.
For each case, firstly we will list some depth $0$ NMS flows.
Then one hand, we will show that such a list is complete
for each case. The main tools here we use are the theorems in  the last three sections. But sometimes, more detailed arguments are also needed.  Two parameters are essential: one is the orientation of the periodic orbits and the other is the accompany graphs.
On the other hand,  we will use some  topologically equivalent invariants (for instance, indexed links) to show that  the flows in the list are pairwise different.

\subsection{2 periodic orbits} \label{2per}

\begin{figure}[htp]
\begin{center}
  \includegraphics[totalheight=4.2cm]{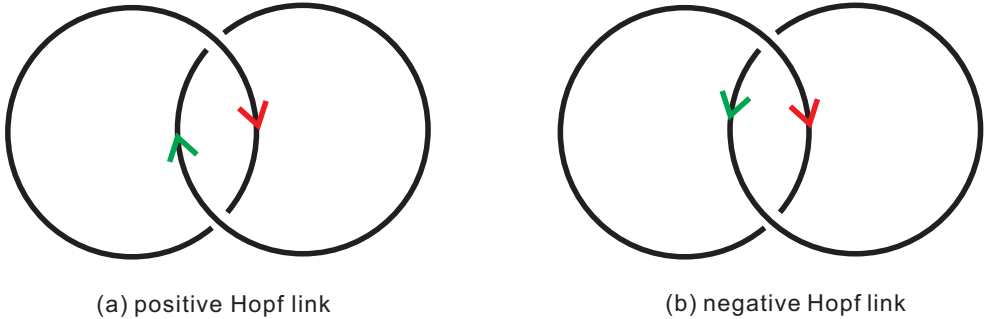}\\
  \caption{Hopf link}\label{hopflink}
  \end{center}
\end{figure}

\begin{lemma}\label{X}
There exist two NMS flows  $\phi_1^t$ and  $\phi_2^t$ on $S^3$ and  a homeomorphism $f:S^3 \rightarrow S^3$ such that they satisfy the following three conditions.
                \begin{enumerate}
                  \item The number of periodic orbits of either $\phi_1^t$ or $\phi_2^t$ is $2$.
                  \item The periodic orbits of  $\phi_1^t$ and $\phi_2^t$ form a positive and a negative Hopf links correspondingly.
                  \item $f$ is a topologically equivalent homeomorphism of $\phi_1^t$  and $\phi_2^t$ by sending the flowlines of $\phi_1^t$ to the flowlines of $\phi_2^t$.
                \end{enumerate}
\end{lemma}

\begin{proof}
$S^3$ can be parameterized as $\{(z_1, z_2)\mid z_1, z_2 \in \mathbb{C}, |z_1|^2 + |z_2|^2 =1 \}$. We define $V_1= \{(z_1,z_2)\mid (z_1,z_2)\in S^3, |z_2|\leq \frac{\sqrt{2}}{2}\}$, $V_2= \{(z_1,z_2)\mid (z_1,z_2)\in S^3, |z_2|\geq \frac{\sqrt{2}}{2}\}$ and $T=V_1 \cap V_2= \{(z_1,z_2)\mid (z_1,z_2)\in S^3, |z_2| = \frac{\sqrt{2}}{2}\}$.
            Denote $c_1 = \{(z_1,z_2)\mid (z_1,z_2)\in S^3, z_2 =0\}$ and $c_2 = \{(z_1,z_2)\mid (z_1,z_2)\in S^3, z_1 =0\}$ to be the center of $V_1$ and $V_2$ correspondingly. Moreover, we parameterize $z_j =\rho_j e^{i \theta_j}$ ($j\in \{1,2\}$) and  endow $c_1$ and $c_2$ the orientations by the natural orientation of $\theta_j$.  In addition, we suppose $\overline{c_j}$ is $c_j$ with opposite orientation. It is easy to check that $c_1 \sqcup c_2$ is a positive Hopf link and $c_1 \sqcup \overline{c_2}$ is a negative Hopf link.

 We suppose $f:S^3 \rightarrow S^3$ such that $f(z_1,z_2)=(z_1, \overline{z_2})$. Obviously $f$ is an opposite orientation self-homeomorphism of $S^3$. One can easily provide an NMS flow $\phi_1^t$ provided by $C^1$ vector field $X^1$ such that $\phi_1^t$ satisfies the following conditions.
            \begin{enumerate}
              \item the periodic orbits of $\phi_1^t$ is $c_1 \sqcup c_2$ which is a positive Hopf link of $S^3$.
              \item $X^1$ is orthogonal to $T$.
            \end{enumerate}

Set $X^2=Df(X^1)$ and $\phi_2^t$ is the flow provided by  $X^2$. Then obviously $\phi_2^t$ is an NMS flow with
periodic orbits $c_1 \sqcup \overline{c_2}$ which is a negative Hopf link of $S^3$. Obviously, $\phi_1^t$, $\phi_t^2$ and $f$ satisfy the three conditions above.
\end{proof}

We denote $\phi_1^t$ by $X$, then we have the following proposition.

\begin{proposition}\label{2periorb}
Up to topological equivalence, there exists exactly  one NMS flow $X$ whose periodic orbits are composed of an attractor $A$ and a repeller $R$. Moreover, the periodic orbits $A \sqcup R$ form a Hopf link in $S^3$.
\end{proposition}
\begin{proof}
 The proposition is a consequence of the following two facts and Lemma \ref{X}.
\begin{itemize}
 \item By Theorem \ref{noV5link}, the indexed link of one of such kind of flows is either a positive or a negative Hopf link, see Figure \ref{hopflink}.
 \item For such kind of NMS flows,  two flows with the same index link (a negative or positive Hopf link) are topologically equivalent. If we  choose some suitable coordinates, this can be followed by Theorem \ref{topequ}.
 \end{itemize}
 \end{proof}

\subsection{3 periodic orbits} \label{3per}
Suppose $\phi_t$ is an NMS flow on $S^3$ with $3$ periodic orbits. Then $\phi_t$ automatically is a depth $0$ NMS flow with three
 periodic orbits: an attractor $A$, a repeller $R$ and a saddle periodic orbit $\gamma$. $\gamma$ has two possibilities  depending on whether $\gamma$ is twisted.

\subsubsection{regular saddle periodic orbit} \label{3reper}
By Proposition \ref{Wadaop},
 when $\gamma$ is a regular saddle periodic orbit.
 $A\sqcup \gamma \sqcup R$ always forms a trivial link, i.e., a three component unlinked, unknotted link. But in this case, there  still exist several different topologically equivalent classes.

   To describe them, let's introduce some more parameters. We choose two small filtrating neighborhoods $N(A)$ and $N(R)$ of $A$ and $R$ correspondingly as follows.
   $W^s (\gamma) \cap \partial N(R)$  is composed of two simple closed curves $w^s_1$ and $w^s_2$ in $\partial N(R)$.  Similarly, $W^u (\gamma) \cap \partial N(A)$ is composed of two  simple closed curves $w^u_1$ and $w^u_2$. $\{w^s_1, w^s_2, w^u_1, w^u_2\}$ is  an accompany graph of $\phi_t$. By 3 in Proposition \ref{filnghd} and Remark \ref{loctopequ}, we can suppose that $w^s_1$ and $w^u_1$ are isotopic to $R$ and $A$ in $N(R)$ and $N(A)$ correspondingly. Moreover, $w^s_2$ and $w^u_2$ are inessential in $\partial N(R)$ and $\partial N(A)$ respectively.

   We denote  $\tau(\phi_t)=(i,j,k)$ ($i,j,k \in \{+,-\}$) as follows. If $w^s_1$ is the same (reps. reverse) orientation with $A$ in $N(A)$, $i=+$ (reps. $i=-$). Similarly, if $w^u_1$ is the same (reps. reverse) orientation with $R$ in $N(R)$, $j=+$ (reps. $j=-$). Moreover, if $w^s_2$ is left-hand (reps. right-hand) orientation in $\partial N(R)$, then $k=+$ (reps. $k=-$).

  By gluing $N(A)$, $N(R)$ and the filtrating neighborhood $N$ of $3$ in Proposition \ref{filnghd}, one can construct  $8$   flows represented by $\phi_t^{\tau}=\phi_t^{(i,j,k)}$ ($i,j,k\in\{+,-\}$).
By Proposition \ref{df}, they are pairwise different up to topological equivalence.  On the other hand, by the definition of coordinates (Section \ref{coordinate}) and Theorem \ref{topequ}, everyone of such kind of flows is topologically equivalent to some $\phi_t^{(i,j,k)}$.

Now we collect these discussions to the following proposition.
\begin{proposition}\label{3regper}
Let $\phi_t$ be an NMS flow on $S^3$ with $3$ periodic orbits such that the saddle periodic orbit is regular. Then $\phi_t$ is topologically
equivalent to  one of $8$ different  flows represented by $\phi_t^{(i,j,k)}$ ($i,j,k\in\{+,-\}$).  Moreover, these $8$ flows are pairwise different,
i.e., not topologically equivalent.
\end{proposition}

\subsubsection{twisted saddle periodic orbit} \label{twi3per}
In the case that $\gamma$ is a twisted saddle periodic orbit,
 one can build $\phi_t$ by gluing the filtrating neighborhoods $N(\gamma)$, $N(A)$ and $N(R)$ together. Then we can observe that $\phi_t$ is  topologically equivalent to one of the following four cases. Moreover, every series of parameters in the following four cases can be realized by some $\phi_t$.
                  \begin{itemize}
                   \item $L(A,\gamma)=2p-1$ ($p\in \mathbb{Z}$) and $L(R,\gamma)=1$. Here $L(A,\gamma)$ is the linking number of $A$ and $\gamma$. Moreover, $A$ is in the same direction of $\gamma$ in the complement of $R$.
                    \item $L(A,\gamma)=2p-1$ ($p\in \mathbb{Z}$) and $L(R,\gamma)=1$. $A$ is in the opposite direction of $\gamma$ in the complement of $R$.
                    \item $L(R,\gamma)=2p-1$ ($p\in \mathbb{Z}$ and $p \neq 0,1$) and $L(A,\gamma)=1$. $R$ is in the same direction of $\gamma$ in the complement of $A$.
                    \item $L(R,\gamma)=2p-1$ ($p\in \mathbb{Z}$ and $p \neq 0,1$) and $L(A,S)=1$.  $R$ is in the opposite direction of $S$ in the complement of $A$.
                  \end{itemize}

 We will show that each set of parameters in the list exactly corresponds to one topologically equivalent class and  they are pairwise different. Firstly, let's explain why the list contains all such kind of flows. The filtrating neighborhood $N(\gamma)$ of $\gamma$ is homeomorphic to $\Sigma(0,2; \frac{1}{2})$. Then for $\phi_t$, $R\cup \gamma$ or $A\cup \gamma$ forms a Hopf link.
 Up to topological equivalence (if it is needed, we can use $f$ constructed in the case $2$ periodic orbits), we suppose the linking number of $R$ (resp. $A$) and $\gamma$, $L(R,\gamma)=1$ (resp. $L(A,\gamma)=1$).  It is easy to know that $L(A,\gamma)=2p-1$ ($p\in \mathbb{Z}$) when $L(R,\gamma)=1$. Similarly, $L(R,\gamma)=2p-1$ when $L(A,\gamma)=1$. Notice that, up to topological equivalence, the case $L(R, \gamma)=\pm 1$ and $L(A,\gamma)=1$ can be included in the case $L(R,\gamma)=1$. Therefore, the list contains all such kind of flows.

  By a few computations and Theorem \ref{topequ}, each set of parameters decides an unique depth $0$ flow on $S^3$.

  In the end, we explain why they are pairwise different. The mapping class group of $S^3$, $\mathcal{MCG} (S^3) \cong Z_2$ and $f$ constructed  in the case $2$ periodic orbits can be regarded as a generator of the group. One can automatically check that the parameters in the first and the third cases of the list are pairwise different under $\mathcal{MCG} (S^3)$. Therefore, the corresponding flows  are pairwise different. To  finish distinguishing all the cases in the list, Proposition \ref{df} is enough.

 We can collect these discussions to the following proposition.
\begin{proposition}\label{3twper}
Let $\phi_t$ be an NMS flow on $S^3$ with $3$ periodic orbits such that the saddle periodic orbit is twisted. Then $\phi_t$  exactly can  be parameterized
by  the four cases listed above.  Moreover, if $\phi_t^1$ and $\phi_t^2$ are two such kind of flows with different parameters in the list, then
 $\phi_t^1$ and $\phi_t^2$ are not topologically equivalent.
\end{proposition}

\begin{remark}
For every such kind of flow $\phi_t$,  by Theorem \ref{V5link} and the parameters of $\phi_t$, one can easily describe the indexed link of $\phi_t$. We leave it to the reader.
\end{remark}

   \subsection{4 periodic orbits}\label{4per}

    \begin{figure}

   \begin{center}
  \includegraphics[totalheight=4.5cm]{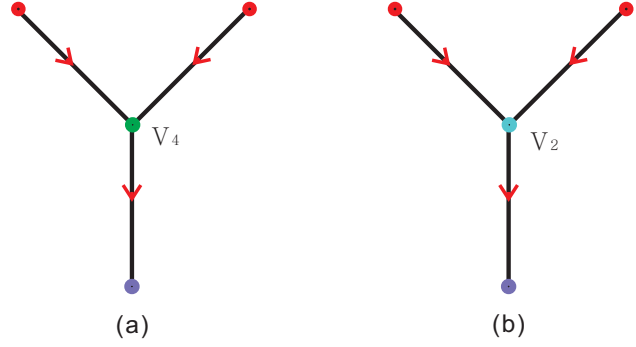}\\
  \caption{four vertices WLG}\label{WLG4}
  \end{center}
   \end{figure}

  Suppose $\phi_t$ is a depth $0$ NMS flow on $S^3$ with $4$ periodic orbits.
  By Theorem \ref{generalWLG}, if $G$ is a neat WLG of $\phi_t$,  the following  two cases never appear.
   \begin{enumerate}
     \item There are two $V_5$ vertices in $G$.
     \item A $V_3$ vertex and a $V_5$ vertex are adjacent in $G$.
   \end{enumerate}

 Then  there exactly exists  one saddle periodic orbit $\gamma$ in $\phi_t$.
For simplify, we will focus on the case that $\phi_t$ contains two repellers $R_1$ and $R_2$. The unique attractor is denoted by $A$.   A Lyapunov graph of $\phi_t$ is one of the two  cases showed in Figure \ref{WLG4}.   Actually, the two cases below and the flows below with converse orientations provide all and pairwise different depth $0$ NMS flows on $S^3$ with $4$ periodic orbits.  In this subsection, we omit some details of the proof. Actually, similar to the discussions above, Theorem \ref{topequ} tell us every case decide an unique topologically equivalent class of this kind of flows and Proposition \ref{df} tell us the flows in the list are pairwise different.

     \begin{enumerate}

       \item $\phi_t$  admits a Lyapunov graph as (a) of Figure \ref{WLG4} shows.  By Theorem \ref{noV5link},  $A\sqcup \gamma \sqcup R_1 \sqcup R_2$ form a four component indexed link which is composed of a Hopf link (which contains $R_1$) and two parallel $(p,q)$ fiber knot of $R_1$ where $p$ and $q$ are coprime.
           Without confusion, we can suppose  $R_1 \sqcup K$ is a positive Hopf link.  Let's describe $(p,q)$ more clearly. The longitude coordinate of $\partial N(R_1)$  is defined such that it bounds a disk in the complement space of $N(R_1)$. $(p,q)$ is defined to be the coordinate of $W^s (\gamma)\cap\partial N(R_1)$ in $\partial N(R_1)$.   The parameters of $\phi_t$ have the following possibilities.
           \begin{itemize}
             \item $K=A$, the orientation of $R_2$ is either positive or negative. Here the orientation of $R_2$ is positive 
             if $R_2$ is the same orientation to $\gamma$ in $S^3 - R_1 \sqcup A$.  Otherwise, the orientation of $R_2$ is negative. 
             \item $K=R_2$, the orientation of $A$ is positive or negative.  The definition of the orientation is similar to the case above. To don't appear any repeat, in this case, we should module more things. First, $|p|\neq 1$ and $|q|\neq 1$.  Secondly,  if two depth $0$ NMS flows have $(p,q)$ and $(q,p)$ parameters correspondingly and the other coordinates are the same, then they are topologically equivalent.
           \end{itemize}

       \item $\phi_t$  admits a Lyapunov graph as (b) of Figure \ref{WLG4} shows. The four component link $A\sqcup \gamma \sqcup R_1 \sqcup R_2$ are composed of a Hopf
       link and two unknotted, unlinked knots where $\gamma$ is an unlinked component. Without confusion, up to topological equivalence, we can suppose that the Hopf link is positive and is composed of $R_1$ and $A$. The following discussion is quite similar to Section \ref{3reper}.
       
       Let's introduce some more parameters. We choose three small filtrating neighborhoods $N(A)$, $N(R_1)$ and $N(R_2)$ of $A$, $R_1$ and $R_2$ correspondingly. Assume that $N(\gamma)= S^3 - N(A)\sqcup N(R_1)\sqcup N(R_2)$ which is a filtrating neighborhood of $\gamma$. The following facts can be followed by Proposition \ref{filnghd} and Remark \ref{loctopequ}.

   $W^s (\gamma) \cap \partial N(R_1)$  is an inessential  simple closed curves $w^s_1$  in $\partial N(R_1)$. 
   $W^s (\gamma) \cap \partial N(R_2)$  is a simple closed curves $w^s_2$ in $\partial N(R_2)$. Moreover, $w^s_2$ is isotopic to $R_2$ in $N(R_2)$.
    $W^u (\gamma) \cap \partial N(A)$ is composed of two  simple closed curves $w^u_1$ and $w^u_2$ in $\partial N(A)$. Here $w^u_1$ is inessential  simple  in $\partial N(A)$ and $w^u_2$  is isotopic to $A$ in $N(A)$.

   $\{w^s_1, w^s_2, w^u_1, w^u_2\}$ is  an accompany graph of $\phi_t$. 
  We denote  $\tau(\phi_t)=(i,j,k)$ ($i,j,k \in \{+,-\}$) as follows. 
  If $w^s_2$ is the same (reps. reverse) orientation with $R_2$ in $N(R_2)$, $i=+$ (reps. $i=-$). 
  Similarly, if $w^u_2$ is the same (reps. reverse) orientation with $A$ in $N(A)$, $j=+$ (reps. $j=-$).
   Moreover, if $w^s_1$ is left-hand (reps. right-hand) orientation in $\partial N(R_1)$, then $k=+$ (reps. $k=-$).
By gluing $N(A)$, $N(R_1)$, $N(R_2)$ and the filtrating neighborhood $N$ of $2$ in Proposition \ref{filnghd}, one can construct  $8$   flows represented by $\varphi_t^{\tau}=\varphi_t^{(i,j,k)}$ ($i,j,k\in\{+,-\}$).
  
  Now we collect these discussions to the following proposition.
\begin{proposition}\label{4perV2}
Let $\phi_t$ be a depth $0$ NMS flow on $S^3$ with $4$ periodic orbits and WLG (b) in Figure \ref{WLG4}.  Then $\phi_t$ is topologically
equivalent to  one of $8$ different  flows represented by $\varphi_t^{(i,j,k)}$ ($i,j,k\in\{+,-\}$).  Moreover, these $8$ flows are pairwise different,
i.e., not topologically equivalent.
\end{proposition}
 \end{enumerate}

{\large \bf Acknowledgements}\\

This work was partially done when the author was visiting Institut de Math\'ematiques de Bourgogne, Universit\'e de Bourgogne, he would like to thank it for its hospitality. He also would like to thank the China Scholarship Council for financial support. The author was supported in part by NSFC.

\appendix

\section{A proof of Theorem \ref{topequ}} \label{appendix}
Before proving the theorem, we introduce an useful lemma. Notice that in the following, for technical reason, we will use  depth $0$ NMS flows of compact 3-manifolds with transverse tori boundary and their  WLG. They are the natural generalizations of standard definitions.

\begin{lemma}\label{glurule}
Suppose $X$ and $Y$ are two depth $0$ NMS flows on two compact 3-manifolds $M$ and $N$ correspondingly. Moreover suppose $T_M \subset \partial^{out} M$ and $T_N \subset \partial^{in} N$.   We denote the union of the saddle periodic orbits  of $X$ and $Y$  by  $\Gamma_X$ and $\Gamma_Y$ correspondingly. Let $h_0: T_M \rightarrow T_N$ and $h_1: T_M \rightarrow T_N$ be two isotopic homeomorphisms satisfying the following conditions.
\begin{enumerate}
  \item $h_i (w^u (\Gamma_X)) \cap w^s (\Gamma_Y) =\emptyset$ ($i\in\{0,1\}$).
  \item $h_0 (w^u (\Gamma_X)) \cap w^s (\Gamma_Y)$ is isotopic to $h_1 (w^u (\Gamma_X)) \cap w^s (\Gamma_Y)$ on $T_N$.
\end{enumerate}
We denote the two new depth $0$ flows on the new 3-manifold $W=M\cup_{h_0} N$ by $Z_0$ and $Z_1$ associated to $h_0$ and $h_1$ correspondingly.
Then $Z_0$ and $Z_1$ are two  depth $0$ NMS flows. Moveover, $Z_0$ and $Z_1$ are topologically equivalent.
\end{lemma}

\begin{proof}
Since $h_0$ and $h_1$ are isotopic and  $h_0 (w^u (\Gamma_X)) \cap w^s (\Gamma_Y)$ is isotopic to $h_1 (w^u (\Gamma_X)) \cap w^s (\Gamma_Y)$.
There exists $H_s : T_M \rightarrow T_N$ ($s\in [0,1]$) which satisfies the following conditions.
\begin{itemize}
  \item $H_0 =h_0$ and $H_1 =h_1$.
  \item $H_s (w^u (\Gamma_X)) \cap w^s (\Gamma_Y)$ is isotopic to $h_0 (w^u (\Gamma_X)) \cap w^s (\Gamma_Y)$ for every $s\in [0,1]$.
\end{itemize}

Suppose the gluing manifold $M\cup_{H_s} N$ is $W$. The reason that $W$ doesn't depend on $H_s$ can be followed by the fact that $H_s$ ($s\in [0,1]$) are pairwise isotopy. Now we define  the new gluing flow $X\cup_{H_s} Y$ by $Z^s$. Automatically, $Z^0 =Z_0$, $Z^1 =Z_1$  and the underline manifold of $Z_s$ is $W$.  To ensure that $Z^s$ is a smooth flow, we can regard that there exist  Riemann metrics on $M$ and $N$ such that $X$ and $Y$ are orthogonal to $T_M$ and $T_N$ correspondingly. Since    $H_s (w^u (\Gamma_X)) \cap w^s (\Gamma_Y)$ is isotopic to $h_0 (w^u (\Gamma_X)) \cap w^s (\Gamma_Y)$ and $X$ and $Y$ are depth $0$ NMS flows, $Z^s$ is a depth $0$ NMS flow on $W$. Therefore, $Z^s$ ($s\in [0,1]$) provides a path in the smooth flow space of $M$ connecting $Z_0$ with $Z_1$. It is well known that  NMS flows are structure stable, therefore $Z_0$ and $Z_1$ are two depth $0$ NMS flows and they are topologically equivalent.
\end{proof}

For a compact graph $G$, a vertex $V\in G$ is called an \emph{isolated inner vertex} if  it satisfies the following two conditions.
 \begin{enumerate}
   \item The degree of $V$, $deg(V)\geq 2$.
   \item  There are  $deg(V)-1$ edges such that each edge is adjacent to $V$ and a degree $1$ vertex.
 \end{enumerate}
  The left unique edge adjacent to $V$ is called the \emph{inner edge} of $V$.

\begin{lemma}\label{isov}
Let $G$ be a compact tree with at least $3$ vertices, then there exists an inner vertex $V\in G$.
\end{lemma}
\begin{proof}
Let's endow $G$ a metric such that the length of each edge is $1$, then naturally  there exists a path $l=(V_0, V_1,\dots,V_n)$ with maximal length among all pathes in $G$. Here $l=(V_0, V_1,\dots,V_n)$ is the path which starts at the vertex $V_0$, passes through $V_1,\dots, V_{n-1}$, and in the end terminates at $V_n$. Then one can easily check that $V_1$ is an isolated inner vertex of $G$.
\end{proof}

\begin{lemma}\label{splitrule}
Under the assumptions of Theorem \ref{topequ}, suppose $V^1 \in G_1$ is an isolated inner vertex with a corresponding inner edge $e^1$, $V^2 =g (V^1) \in G_2$ and $e^2 = g (e^1)$. An inner point of $e^i$ cuts $G_i$ to $G_i^0$ and $G_i^1$. Correspondingly, the transverse tori $T_i$  associated to the inner point cuts $\phi_t^i$ to depth $0$ NMS flows $X_t^i$ and $Y_t^i$ on $M_i$ and $N_i$ correspondingly. Then, two accompany graphs associated to $X_t^1$ and $X_t^2$ (resp. $Y_t^1$ and $Y_t^2$)  are equivalent under $G_1^0$ and $G_2^0$ (resp. $G_1^1$ and $G_2^1$) by $g$.
\end{lemma}

\begin{proof}
Suppose $\mathcal{T}_i^0$  ($i\in \{1,2\}$) is composed of $\partial M_i$ and  the boundary of a small tubular neighborhood of the repellers and the attractors in $M_i$. Similarly, we can suppose $\mathcal{T}_i^1$ is composed of $\partial N_i$ and  the boundary of a small tubular neighborhood of the repellers and the attractors in $N_i$. Pay attention on that originally,  $T_i = \partial M_i = \partial N_i$. We only need to prove that there exist $f_j: \mathcal {T}_1^j\rightarrow \mathcal {T}_2^j$ ($j \in \{0,1\}$) such that $f_j$ preserves two accompany graphs on $\mathcal{T}_1^j$  and $\mathcal{T}_2^j$ and satisfies the conditions in Definition \ref{accequ}. Without confusion, we can suppose that the flow $\phi_t^i$ is transverse outward to $\partial M_i$ on $M_i$.

Now let's construct $f_j$. Firstly, on a connected component $T \in  \mathcal{T}_1^0-\partial M_1$, $f_0 =f$. Similarly, on a connected component $T \in  \mathcal{T}_1^1-\partial N_1$, $f_1 =f$. They automatically satisfy the conditions in Definition \ref{accequ}. Now we only need to focus on the construction of $f_0 |_{\partial M_1}$ and $f_1 |_{\partial N_1}$.

$f_0$ firstly can be naturally defined on $w^u  (\gamma_1^0)$ (to $w^u  (\gamma_2^0)$) induced by the accompany graphs of $\phi_t^1$ and $\phi_t^2$. Since $w^u  (\gamma_1^0)$ is one of the following three cases on the corresponding torus (see Proposition \ref{filnghd} and Section \ref{coordinate}):
    \begin{enumerate}
      \item $w^u  (\gamma_1^0)$ is parallel to the longitude;
      \item $w^u  (\gamma_1^0)$ is parallel to the meridian;
      \item $w^u  (\gamma_1^0)$ is inessential.
    \end{enumerate}
Then, we can easily extend the map to $f_0$ on $\partial {M_1}$ which satisfies  the conditions in Definition \ref{accequ} (in particular, the second condition). Moreover, following  Proposition \ref{filnghd} and Section \ref{coordinate}, one can check that
we can construct $f_0$ such that $f_0 (w^s (\Gamma_1)) =w^s (\Gamma_2)$. Here $\Gamma_1$ and $\Gamma_2$ are the saddle periodic orbits in $N_1$ and $N_2$ correspondingly.

Suppose $h_i$ is the natural gluing map from $\partial {M_i}$ to $\partial {N_i}$ ($i\in \{1,2\}$). We can define $f_1$ on $\partial {N_1}$ by
$f_1 = h_2 \circ f_0 \circ h_1^{-1}$. Since $f_0 (w^s (\Gamma_1)) =w^s (\Gamma_2)$, $f_1 (w^s  (\Gamma_1))=w^s (\Gamma_2)$.
Notice that the gluing matrices $[h_1]$ from $M_1$ to $N_1$ and $[h_2]$ from $M_2$ to $N_2$ are the same, moreover the gluing matrix $[f_0]$ associated to $f_0$ is $\left(
                                                                                     \begin{array}{cc}
                                                                                       1 & 0 \\
                                                                                       0 & 1 \\
                                                                                     \end{array}
                                                                                   \right)$.
Then,  the gluing matrix $[f_1]$ associated to $f_1$ also is $\left(
                                                                                     \begin{array}{cc}
                                                                                       1 & 0 \\
                                                                                       0 & 1 \\
                                                                                     \end{array}
                                                                                   \right)$.
All these arguments and the construction of $f_1$ ensure that $f_1$ satisfies the conditions in Definition \ref{accequ}. Then the lemma is proved.
\end{proof}

 After these preparations, we can finish the proof of  Theorem \ref{topequ}.

\begin{proof}
(\emph{The proof of Theorem \ref{topequ}})
Although the theorem is stated for depth $0$ NMS flows on $S^3$. For the convenience to do some inductive arguments, we also consider depth $0$ NMS flows on compact 3-manifolds with transverse tori as the boundaries.

 Suppose $M$ and $N$ are two compact 3-manifolds and there are depth $0$ NMS flows  $\psi_t^1$ and $\psi_t^2$ on $M$ and $N$ satisfying the conditions in the theorem.  If $M$  doesn't admit any saddle periodic orbit, by  Lemma \ref{glurule}, $\psi_t^1$ and $\psi_t^2$ are topologically equivalent.  If $M$ admits  one saddle periodic orbit.  By Proposition \ref{filnghd} and Remark \ref{loctopequ}, the two saddle filtrating neighborhoods of  $\psi_t^1$ and $\psi_t^2$ respectively are  topologically equivalent. Moreover, by  Lemma \ref{glurule}, $\psi_t^1$ and $\psi_t^2$ are topologically equivalent.

  Now we suppose the number of the periodic orbits of $\phi_t^1$ is $k$ and $k\geq3$. Moveover we assume that for each depth $0$ NMS flow with periodic orbits number  strictly less than $k$, the theorem is correct.
By Lemma \ref{isov}, there exists two isolated inner vertices $v^1\in G_1$ and $v^2 = g(v^1) \in G_2$ correspondingly.
 One can cut $G_1$ along  the inner edge  associated to $v^1$ to $G_1^0$ and $G_1^1$ such that $G_1^0$ contains the isolated inner vertex. Similarly,  $G_2$ is cut to $G_2^0$ and $G_2^1$. Then we can use the notations and consequences in Lemma \ref{splitrule}.

 By Lemma \ref{splitrule} and the inductive assumption, $X_t^1$ and $Y_t^1$ are topologically equivalent to $X_t^2$ and $Y_t^2$ by $g_0$ and $g_1$ correspondingly. Suppose the natural gluing map from $M_i$ to $N_i$ is  $h_i$ and $h_2' =g_1^{-1} \circ h_2 \circ g_0$.
One can check that  $h_1$ and $h_2'$ are isotopic and they satisfy the following conditions.
\begin{enumerate}
  \item $h_1 (w^u (\gamma_1^0)) \cap w^s (\Gamma_1) =\emptyset$ and $h_2' (w^u (\gamma_1^0)) \cap w^s (\Gamma_1) =\emptyset$.
  \item $h_1 (w^u (\Gamma_X)) \cap w^s (\Gamma_Y)$ is isotopic to $h_2' (w^u (\Gamma_X)) \cap w^s (\Gamma_Y)$ on $T_N$.
\end{enumerate}
Therefore, by Lemma \ref{glurule}, $\phi_t^1$ and $\phi_t^2$ are topologically equivalent.
\end{proof}

\vskip 1cm

\noindent Bin Yu

\noindent {\small Department of Mathematics}

\noindent{\small Tongji University, Shanghai 2000
92, CHINA}

\noindent{\footnotesize{E-mail: binyu1980@gmail.com }}

\end{document}